\documentclass[12pt]{article}%
\usepackage{amsmath}
\usepackage{amsfonts}
\usepackage{mitpress}
\usepackage{amssymb}
\usepackage{graphicx}%
\providecommand{\U}[1]{\protect\rule{.1in}{.1in}}
\newtheorem{theorem}{Theorem}

\newtheorem{corollary}[theorem]{Corollary}

\newtheorem{definition}[theorem]{Definition}

\newtheorem{lemma}[theorem]{Lemma}

\newtheorem{remark}[theorem]{Remark}

\newenvironment{proof}[1][Proof]{\noindent\textbf{#1.} }{\ \rule{0.5em}{0.5em}}
\newdimen\dummy
\dummy=\oddsidemargin
\begin{document}

\title{Structure and regularity of the global attractor of a reaction-diffusion
equation with non-smooth nonlinear term}
\author{Oleksiy V. Kapustyan $^{1}$, Pavlo O. Kasyanov $^{2}$, Jos\'{e} Valero $^{3}$\\{\small ${}^{1}$~Kyiv National Taras Shevchenko University,} {\small Kyiv,
Ukraine. \textit{E-mail: alexkap@univ.kiev.ua}}\\{\small ${}^{2}$Institute for Applied System Analysis at National Technical
University of Ukraine "KPI" of NAS of Ukraine,} {\small Kyiv, Ukraine.
\textit{E-mail: kasyanov@i.ua} }\\{\small ${}^{3}$~Universidad Miguel Hernandez de Elche, }{\small Centro de
Investigaci\'{o}n Operativa, Avda. Universidad s/n, }{\small 03202}%
-{\small Elche, Spain. \textit{E-mail: jvalero@umh.es\bigskip\bigskip}}}
\maketitle

\begin{abstract}
In this paper we study the structure of the global attractor for a
reaction-diffusion equation in which uniqueness of the Cauchy problem is not
guarantied. We prove that the global attractor can be characterized using
either the unstable manifold of the set of stationary points or the stable one
but considering in this last case only solutions in the set of bounded
complete trajectories.

\end{abstract}

\bigskip\bigskip



\textbf{Short title: }Structure of the attractor of reaction-diffusion equations

\bigskip

\section{Introduction}

In this paper we study the structure of the global attractor of a
reaction-diffusion equation in which the nonlinear term satisfy suitable
growth and dissipative conditions, but there is no condition ensuring
uniqueness of the Cauchy problem (like e.g. a monotonicity assumption). Such
equation generates in the general case a multivalued semiflow having a global
compact attractor (see \cite{Kap1}, \cite{KapustValero06}). Also, it is known
\cite{KMVY} that the attractor is the union of all bounded complete
trajectories of the semiflow.

If we study the global attractor in more detail we can get a better
understanding of \ the dynamics of the semiflow by restricting our attention
inside the attractor. In particular, it is important to establish the
relationship between the attractor and the stable and unstable manifolds of
the set of stationary points. In the single-valued case, when for example the
nonlinear term is a polynomial or its derivative satisfies some assumptions,
it is well known \cite{BabinVishik85}, \cite{BabinVishik89}, \cite{Temam2}
that the attractor is the unstable manifold of the set of stationary points.
Moreover, if the set of stationary points is discrete, then it is the union of
all heteroclinic orbits connecting the stationary points. In more particular
parabolic equations the structure of the attractor has been completely
understood by obtaining a list of which stationary points are joined to each
other. This the case of the famous Chafee-Infante equation \cite{Henry85} or
general scalar parabolic equations under suitable restrictions
\cite{Fiedler85}, \cite{Fiedler96}, \cite{Rocha88}, \cite{Rocha91}. Also, in
\cite{Hale} similar results are obtained for retarded differential equations.

In \cite{ArrRBVal} the structure of the global attractor of a scalar parabolic
differential inclusion generating a multivalued semiflow is studied, obtaining
a partial description about which pairs of stationary points are joined. As
far as we know this is the only published result about the heteroclinic
connections between stationary points in the multivalued case.

Let $\mathbb{F}$ be the set of all complete trajectories, $\mathbb{K}$ be the
set of all bounded complete trajectories and $\mathfrak{R}$ the set of
equilibria. We define the sets%
\[%
\begin{array}
[c]{c}%
M^{-}(\mathfrak{R})=\left\{  z\,:\,\exists\gamma(\cdot)\in\mathbb{K}%
,\,\ \gamma(0)=z,\,\,\,\ \mathrm{dist}_{L^{2}(\Omega)}(\gamma(t),\mathfrak{R}%
)\rightarrow0,\,\ t\rightarrow+\infty\right\}  ,\\
M^{+}(\mathfrak{R})=\left\{  z\,:\,\exists\gamma(\cdot)\in\mathbb{F}%
,\,\ \gamma(0)=z,\,\,\,\ \mathrm{dist}_{L^{2}(\Omega)}(\gamma(t),\mathfrak{R}%
)\rightarrow0,\,\ t\rightarrow-\infty\right\}  .
\end{array}
\]
We prove in this paper that the global attractor of a rather general
reaction-diffusion equation without uniqueness can be described in terms of
either $M^{+}(\mathfrak{R})$ or $M^{-}(\mathfrak{R})$, that is, the unstable
manifold of the set of stationary points or the stable one but considering in
this last case only solutions in the set of bounded complete trajectories.

In Section \ref{StrWeak} it is proved that the attractor of the multivalued
semiflow generated by weak solutions in the phase space $L^{2}\left(
\Omega\right)  $ is the closure of $M^{-}(\mathfrak{R})$. Also, $M^{+}%
(\mathfrak{R})$ is contained in the attractor, and coincides with it when
uniqueness takes place for regular initial data.

In Section \ref{StrReg} we consider the multivalued semiflow generated by
regular solutions, which are the weak solutions which become strong after an
arbitrary small time. We prove first the existence of a global attractor in
the phase space $L^{2}\left(  \Omega\right)  $ which is, moreover, compact in
$H_{0}^{1}\left(  \Omega\right)  $. After that we establish that it coincides
with the unstable manifold of the set of stationary points, and also with the
stable one when we consider only bounded complete solutions.

In Section \ref{StrStrong} we consider the multivalued semiflow generated by
strong solutions. We prove first the existence of a global attractor in the
phase space $H_{0}^{1}\left(  \Omega\right)  $ and that the attractors of the
regular and strong cases coincide. Finally, the same result about the
structure of the attractor as in the case of regular solutions is given.

\section{Setting of the problem}

In a bounded domain $\Omega\subset\mathbb{R}^{3}$ with sufficiently smooth
boundary $\partial\Omega$ we consider the problem
\begin{equation}
\left\{
\begin{array}
[c]{l}%
u_{t}-\Delta u+f(u)=h,\quad x\in\Omega,\ t>0,\\
u|_{\partial\Omega}=0,
\end{array}
\right.  \label{1}%
\end{equation}
where
\begin{equation}%
\begin{array}
[c]{c}%
h\in L^{2}(\Omega),\ \\
f\in C(\mathbb{R}),\\
|f(u)|\leq C_{1}(1+|u|^{3}),\quad\forall u\in\mathbb{R},\\
f(u)u\geq\alpha|u|^{4}-C_{2},\quad\forall u\in\mathbb{R},
\end{array}
\label{2}%
\end{equation}
with $C_{1},C_{2},\alpha>0$.

We denote by $A$ the operator $-\Delta$ with Dirichlet boundary conditions, so
that $D\left(  A\right)  =H^{2}\left(  \Omega\right)  \cap H_{0}^{1}\left(
\Omega\right)  .$ As usual, denote the first eigenvalue of $A$ by $\lambda
_{1}.$

Denote $F(u)=\int_{0}^{u}f(s)ds$. From (\ref{2}) we have that $\liminf
\limits_{|u|\rightarrow\infty}\frac{f(u)}{u}=\infty$, and for some
$D_{1},D_{2},\delta>0,$%
\begin{equation}
|F(u)|\leq D_{1}(1+|u|^{4}),\,\ F(u)\geq\delta u^{4}-D_{2},\quad\forall
u\in\mathbb{R}. \label{PropF}%
\end{equation}

The function $u\in L_{loc}^{2}(0,+\infty;H_{0}^{1}(\Omega))\bigcap L_{loc}%
^{4}(0,+\infty;L^{4}(\Omega))$ is called a weak solution of (\ref{1}) on
$(0,+\infty)$ if for all $T>0\,,\ v\in H_{0}^{1}\left(  \Omega\right)
,\,\eta\in C_{0}^{\infty}(0,T),$
\begin{equation}
-\int\limits_{0}^{T}(u,v)\eta_{t}dt+\int\limits_{0}^{T}\left(  (u,v)_{H_{0}%
^{1}(\Omega)}+(f(u),v)-(h,v)\right)  \eta dt=0, \label{3}%
\end{equation}
where $\Vert\cdot\Vert,\,\ (\cdot,\cdot)$ are the norm and the scalar product
in $L^{2}(\Omega)$. We denote by $\Vert\cdot\Vert_{X}$ the norm in the
abstract Banach space $X$, whereas $(\cdot,\cdot)_{H}$ will be the scalar
product in the abstract Hilbert space $H$. Also, $P\left(  X\right)  $ will be
the set of all non-empty subsets of $X.$

It is well known \cite[Theorem 2]{ACRV2010} or \cite[p.284]{ChepVishikBook}
that for any $u_{0}\in L^{2}(\Omega)$ there exists at least one weak solution
of (\ref{1}) with $u(0)=u_{0}$ (and it may be non unique) and that any weak
solution of (\ref{1}) belongs to $C\left(  [0,+\infty);L^{2}(\Omega)\right)
$. Moreover, the function $t\mapsto\Vert u(t)\Vert^{2}$ is absolutely
continuous and
\begin{equation}
\frac{1}{2}\frac{d}{dt}\Vert u(t)\Vert^{2}+\Vert u(t)\Vert_{H_{0}^{1}(\Omega
)}^{2}+\left(  f(u(t)),u(t)\right)  -(h,u(t))=0\text{ a.e.} \label{4}%
\end{equation}
We define
\begin{equation}%
\begin{array}
[c]{c}%
K^{+}=\left\{  u(\cdot)\,:\,u(\cdot)\,\text{is a weak solution of (\ref{1}%
)}\right\}  ,\\
G:\mathbb{R}^{+}\times L^{2}(\Omega)\rightarrow P(L^{2}(\Omega)),\\
G(t,u_{0})=\left\{  u(t)\,:\,u(\cdot)\in K^{+},\,\ u(0)=u_{0}\right\}  .
\end{array}
\label{5}%
\end{equation}

\begin{definition}
Let $X$ be a complete metric space. The multivalued map $G:\mathbb{R}%
^{+}\times X\rightarrow P(X)$ is a multivalued semiflow (m-semiflow) if:

\begin{enumerate}
\item $G(0,u_{0})=u_{0},\,\ \forall u_{0}\in X;$

\item $G(t+s,u_{0})\subset G(t,G(s,u_{0})),$ $\forall\ t,s\geq0$, $\forall
u_{0}\in X$.
\end{enumerate}

It is called strict if $G(t+s,u_{0})=G(t,G(s,u_{0})),$ $\forall\ t,s\geq0$,
$\forall u_{0}\in X$.
\end{definition}

\begin{definition}
\label{defi:glatr} The set $\Theta\subset X$ is called a global attractor of
$G$, if:

\begin{enumerate}
\item $\Theta\subset G\left(  t,{\Theta}\right)  ,$ $\forall t\geq0$
(negatively semi-invariance);

\item for any bounded set $B\subset X$,%
\begin{equation}
\mathrm{dist}_{X}(G(t,B),\Theta)\rightarrow0,\ \text{as }t\rightarrow+\infty,
\label{Attraction}%
\end{equation}
where
\[
\mathrm{dist}_{X}(A,B)=\sup\limits_{x\in A}\inf\limits_{y\in B}\Vert
x-y\Vert_{X}.
\]

\item It is mininal, gthat is, for any closed set $C$ satisfying
(\ref{Attraction}) it holds $\Theta\subset C.$
\end{enumerate}

The global attractor is called invariant if $\Theta=G\left(  t,{\Theta
}\right)  ,$ $\forall t\geq0.$
\end{definition}

The map $G$ defined by (\ref{5}) is a strict multivalued semiflow which
possesses a global compact invariant connected attractor \cite{Kap1},
\cite{KapustValero06}, \cite{KapVal09}. Our aim is to give a characterization
of the attractor. First we shall define complete trajectories for problem
(\ref{1}).

\begin{definition}
\label{defi:traject}The map $\gamma:\mathbb{R}\rightarrow L^{2}(\Omega)$ is
called a complete trajectory of $K^{+}$ if
\[
\,\,\,\ \gamma(\cdot+h)|_{[0,+\infty)}\in K^{+},\ \forall h\in\mathbb{R},
\]
that is, if $\gamma|_{[\tau,+\infty)}$ is weak solution of (\ref{1}) on
$(\tau,+\infty),\,\ \forall\tau\in\mathbb{R}$.

In the following section (see Theorem \ref{PropK+}) it will be shown that this
is equivalent to the following:
\[
\ \gamma(t+s)\in G(t,\gamma(s)),\ \forall t\geq0,\,s\in\mathbb{R}.
\]

\end{definition}

Let $\mathbb{K}$ be the set of all bounded (in the $L^{2}(\Omega)$ norm)
complete trajectories. It is known that the global attractor of $G$ is the
union of all bounded complete trajectories. We recall that the global
attractor $\Theta$ is called stable if for any $\epsilon>0$ there exists
$\delta>0$ such that
\[
G(t,O_{\delta}(\Theta))\subset O_{\epsilon}(\Theta)\text{, }\forall\ t\geq0.
\]

\begin{theorem}
\label{teor1} \cite[Theorem 3.18]{KMVY} Under conditions (\ref{2}) the
m-semiflow (\ref{5}) has a global compact invariant attractor $\Theta\subset
L^{2}(\Omega)$ which is connected, stable and
\begin{equation}
\Theta=\left\{  \gamma(0)\,:\,\gamma(\cdot)\in\mathbb{K}\right\}
=\bigcup\limits_{t\in\mathbb{R}}\left\{  \gamma(t)\,:\,\gamma(\cdot
)\in\mathbb{K}\right\}  . \label{6}%
\end{equation}

\end{theorem}

Let $\mathfrak{R}$ be the set of all stationary points of (\ref{1}), i.e., the
points $u\in H_{0}^{1}\left(  \Omega\right)  $ such that%
\begin{equation}
-\Delta u+f(u)=h\text{ in }H^{-1}\left(  \Omega\right)  , \label{Stationary}%
\end{equation}
and%
\begin{equation}%
\begin{array}
[c]{c}%
M^{-}(\mathfrak{R})=\left\{  z\,:\,\exists\gamma(\cdot)\in\mathbb{K}%
,\,\ \gamma(0)=z,\,\,\,\ \mathrm{dist}_{L^{2}(\Omega)}(\gamma(t),\mathfrak{R}%
)\rightarrow0,\,\ t\rightarrow+\infty\right\}  ,\\
M^{+}(\mathfrak{R})=\left\{  z\,:\,\exists\gamma(\cdot)\in\mathbb{F}%
,\,\ \gamma(0)=z,\,\,\,\ \mathrm{dist}_{L^{2}(\Omega)}(\gamma(t),\mathfrak{R}%
)\rightarrow0,\,\ t\rightarrow-\infty\right\}  .
\end{array}
\label{M}%
\end{equation}

It is known \cite{Temam2}, \cite[p.106]{BabinVishik89}, \cite{BabinVishik85}%
\textbf{\ }that under additional conditions on $f$, like
\begin{align}
f  &  =\sum\limits_{i=0}^{3}\alpha_{i}u^{i},\,\ \alpha_{3}>0,\,\ \label{7}\\
f  &  \in C^{1}(\mathbb{R})\text{ and }\exists C_{3}>0\text{ such that
}\left\vert f^{\prime}(u)\right\vert \leq C_{3}\left\vert u\right\vert
^{\frac{4}{3}},\ \forall u\in\mathbb{R},\nonumber\\
\text{or }f  &  \in C^{1}(\mathbb{R})\text{ and }\exists C_{3}>0\text{ such
that }f^{\prime}\geq-C_{3},\nonumber
\end{align}
$G$ is a single-valued semigroup, the set $\Theta$ is bounded in $H^{2}%
(\Omega)\bigcap H_{0}^{1}(\Omega)$ and
\[
\Theta=M^{+}(\mathfrak{R})
\]
Moreover, in \cite[p.106]{BabinVishik89} it is proved that
\begin{equation}
\Theta=M^{+}(\mathfrak{R})=M^{-}(\mathfrak{R}). \label{8}%
\end{equation}

$M^{+}(\mathfrak{R})$ is the unstable set of $\mathfrak{R}$. We note that
under conditions (\ref{7}) attraction takes place in $H_{0}^{1}(\Omega)$. We
observe that in this case an equivalent definition of the set $M^{+}%
(\mathfrak{R})$ is the following
\[
M^{+}(\mathfrak{R})=\left\{  z\,:\,\exists\gamma(\cdot)\in\mathbb{K}%
,\,\ \gamma(0)=z,\,\,\,\ \mathrm{dist}_{L^{2}(\Omega)}(\gamma(t),\mathfrak{R}%
)\rightarrow0,\,\ t\rightarrow-\infty\right\}  ,
\]
as for every complete trajectory $\gamma\left(  \text{\textperiodcentered
}\right)  \in\mathbb{F}$ as in (\ref{M}) we have that the set $\cup
_{t\in\mathbb{R}}\gamma\left(  t\right)  $ is bounded, so that the inclusion
$\gamma\left(  \text{\textperiodcentered}\right)  \in\mathbb{K}$ follows.

The aim of our paper is to obtain something like (\ref{8}) for $K^{+}$ under
the general conditions (\ref{2}). Moreover, taking a more regular set of
solutions we will show that the equality (\ref{8}) holds.

\section{About some properties of complete trajectories and fixed points of
m-semiflows}

We will prove in this section some useful properties of fixed points and
complete trajectories for abstract multivalued semiflows.

Consider a complete metric space $X$ and let
\[
W^{+}=C(\mathbb{R}^{+};X).
\]
Let $\mathcal{R}\subset W^{+}$ be some set of functions such that the
following conditions hold:

\begin{enumerate}
\item[$\left(  K1\right)  $] For any $x\in X$ there exists $\varphi
\in\mathcal{R}$ such that $\varphi\left(  0\right)  =x.$

\item[$\left(  K2\right)  $] $\varphi_{\tau}\left(  \text{\textperiodcentered
}\right)  =\varphi\left(  \text{\textperiodcentered}+\tau\right)
\in\mathcal{R}$ for any $\tau\geq0$, $\varphi\left(  \text{\textperiodcentered
}\right)  \in\mathcal{R}$ (translation property).
\end{enumerate}

Consider also some additional assumptions, which will be needed in order to
obtain good properties. Namely:

\begin{enumerate}
\item[$\left(  K3\right)  $] Let $\varphi_{1},\varphi_{2}\in\mathcal{R}$ be
such that $\varphi_{2}(0)=\varphi_{1}(s)$, where $s>0$. Then the function
$\varphi\left(  \text{\textperiodcentered}\right)  ,$ defined by%
\[
\varphi(t)=\left\{
\begin{array}
[c]{c}%
\varphi_{1}\left(  t\right)  \text{ if }0\leq t\leq s,\\
\varphi_{2}\left(  t-s\right)  \text{ if }s\leq t,
\end{array}
\right.
\]
belongs to $\mathcal{R}$ (concatenation property).

\item[$\left(  K4\right)  $] For any sequence $\varphi^{n}\left(
\text{\textperiodcentered}\right)  \in\mathcal{R}$ such that $\varphi
^{n}\left(  0\right)  \rightarrow\varphi_{0}$ in $X$, there exists a
subsequence $\varphi^{n_{k}}$ and $\varphi\in\mathcal{R}$ such that
\[
\varphi^{n_{k}}\left(  t\right)  \rightarrow\varphi\left(  t\right)  \text{,
}\forall t\geq0.
\]

\end{enumerate}

We define the multivalued map $G:\mathbb{R}^{+}\times X\rightarrow P(X)$ in
the following way:
\[
y\in G\left(  t,x\right)  \ \text{if}\ \exists\ \varphi\in\mathcal{R}%
\ \text{such that }y=\varphi\left(  t\right)  ,\ \varphi\left(  0\right)  =x.
\]

\begin{lemma}
\label{MS}\cite[Lemma 9]{KapPanVal} Let $\left(  K1\right)  -\left(
K2\right)  $ hold. Then $G$ is a multivalued semiflow. Moreover, if $\left(
K3\right)  $ is true, then $G$ is a strict m-semiflow.
\end{lemma}

We define now the concept of fixed point and complete trajectory for
$\mathcal{R}$.

\begin{definition}
\label{DefFixed}The point $z\in X$ is a fixed point of $\mathcal{R}$, if
$\varphi\left(  t\right)  \equiv z\in\mathcal{R}.$ The set of all fixed points
will be denoted by $\mathfrak{R}_{\mathcal{R}}\mathfrak{.}$

The map $\gamma:\mathbb{R}\rightarrow X$ is called a complete trajectory of
$\mathcal{R}$ if
\[
\,\,\,\ \gamma(\cdot+h)|_{[0,+\infty)}\in\mathcal{R},\ \forall h\in
\mathbb{R}.
\]

\end{definition}

We will show that the fixed points of $\mathcal{R}$ coincide with the
stationary points of $G$ under assumptions $\left(  K1\right)  -\left(
K4\right)  $.

\begin{lemma}
\label{FixedPointsCar}Let $\left(  K1\right)  -\left(  K2\right)  $ hold. Then
$z\in\mathfrak{R}_{\mathcal{R}}$ implies $z\in G\left(  t,z\right)  $ for all
$t\geq0.$

Let $\left(  K1\right)  -\left(  K4\right)  $ hold. Then $z\in\mathfrak{R}%
_{\mathcal{R}}$ if and only if $z\in G\left(  t,z\right)  $ for all $t\geq0.$
\end{lemma}

\begin{proof}
When $\left(  K1\right)  -\left(  K2\right)  $ hold, if $z\in\mathfrak{R}%
_{\mathcal{R}}$, it is obvious that $z\in G\left(  t,z\right)  $ for all
$t\geq0.$

Conversely, let $\left(  K1\right)  -\left(  K4\right)  $ hold and let $z\in
G\left(  t,z\right)  $ for all $t\geq0.$ This means that for any $T>0$ there
exists $u^{T}\in\mathcal{R}$ such that $u^{T}\left(  T\right)  =z$ and
$u^{T}\left(  0\right)  =z.$ Consider in the interval $[0,n]$ the $n$-dyadic
partition $D_{n}:=\{j2^{-n}\,:\,j=0,1,2,...,n2^{n}\}$. In each interval
$[j2^{-n},\left(  j+1\right)  2^{-n}]$, $j=0,1,...,n2^{n}-1,$ we consider
$u_{j}^{n}\in\mathcal{R}$ such that $u_{j}^{n}\left(  0\right)  =u_{j}%
^{n}\left(  2^{-n}\right)  =z$. Then we take the concatenation of all these
functions
\[
u^{n}\left(  t\right)  =\left\{
\begin{array}
[c]{cc}%
u_{j}^{n}\left(  t-j2^{-n}\right)  \text{,} & \text{if }j2^{-n}\leq
t\leq\left(  j+1\right)  2^{-n}\text{, }j=0,1,...,n2^{n}-1,\\
\varphi\left(  t-n\right)  , & \text{if }t\geq n,
\end{array}
\right.
\]
where $\varphi\in\mathcal{R},$ with $\varphi\left(  0\right)  =z,$ is
arbitrary. Thus%
\[
u^{n}\left(  t\right)  =z\text{ for all }t\in D_{n}.
\]
By conditions $\left(  K3\right)  -\left(  K4\right)  $ we have that $u^{n}$
belongs to $\mathcal{R}$ and the existence of $u\in\mathcal{R}$ and a
subsequence $u^{n_{k}}$ such that
\[
u^{n_{k}}\left(  t\right)  \rightarrow u\left(  t\right)  \text{ in }X\text{
for all }t\geq0.
\]
Let $D=\cup_{n\in\mathbb{N}}D_{n}$. Hence, $u\left(  t\right)  =z$ for all
$t\in D$. Since $u\in C([0,+\infty),X)$, we obtain that $u\left(  t\right)
=z$ for all $t\geq0$, so that $z\in\mathfrak{R}_{\mathcal{R}}.$
\end{proof}

\bigskip

We will show further the relation between complete trajectories of
$\mathcal{R}$ and $G.$

\begin{lemma}
\label{CompleteTrajEquiv}If $\left(  K1\right)  -\left(  K4\right)  $ hold,
then the map $\gamma:\mathbb{R}\rightarrow X$ is a complete trajectory of
$\mathcal{R}$ if and only if
\begin{equation}
\gamma\left(  t+s\right)  \in G\left(  t,\gamma\left(  s\right)  \right)
\text{ for all }s\in\mathbb{R}\text{ and }t\geq0. \label{PropCompleteTraj}%
\end{equation}
When $\left(  K1\right)  -\left(  K2\right)  $ hold, then any complete
trajectory of $\mathcal{R}$ satisfies (\ref{PropCompleteTraj})
\end{lemma}

\begin{proof}
Under conditions $\left(  K1\right)  -\left(  K2\right)  $ it is obvious that
any complete trajectory of $\mathcal{R}$ satisfies (\ref{PropCompleteTraj}).

Assume $\left(  K1\right)  -\left(  K4\right)  $. Conversely, let
$\gamma\left(  \text{\textperiodcentered}\right)  $ satisfy
(\ref{PropCompleteTraj}). Consider in the interval $[0,n]$ the $n$-dyadic
partition $D_{n}:=\{j2^{-n}\,:\,j=0,1,2,...,n2^{n}\}$. Let $\tau\in\mathbb{R}$
be arbitrary. In each interval $[j2^{-n},\left(  j+1\right)  2^{-n}]$,
$j=0,1,...,n2^{n}-1,$ we consider $u_{j}^{n}\in\mathcal{R}$ such that
$u_{j}^{n}\left(  0\right)  =\gamma\left(  \tau+j2^{-n}\right)  $, $u_{j}%
^{n}\left(  2^{-n}\right)  =\gamma\left(  \tau+(j+1)2^{-n}\right)  $. We take
the concatenation of all these functions
\[
u^{n}\left(  t\right)  =\left\{
\begin{array}
[c]{cc}%
u_{j}^{n}\left(  t-j2^{-n}\right)  \text{,} & \text{if }j2^{-n}\leq
t\leq\left(  j+1\right)  2^{-n}\text{, }j=0,1,...,n2^{n}-1,\\
\varphi\left(  t-n\right)  , & \text{if }t\geq n,
\end{array}
\right.
\]
where $\varphi\in\mathcal{R},$ with $\varphi\left(  0\right)  =\gamma\left(
\tau+n\right)  ,$ is arbitrary.

Then $u^{n}\in\mathcal{R}$ by $\left(  K3\right)  $ and
\[
u^{n}\left(  t\right)  =\gamma\left(  t+\tau\right)  \text{ for all }t\in
D_{n}.
\]
In view of $\left(  K4\right)  $ there exists $u\in\mathcal{R}$ and a
subsequence $u^{n_{k}}$ such that
\[
u^{n_{k}}\left(  t\right)  \rightarrow u\left(  t\right)  \text{ in }X\text{
for all }t\geq\tau.
\]
Let $D=\cup_{n\in\mathbb{N}}D_{n}$. Hence, $u\left(  t\right)  =\gamma\left(
t+\tau\right)  $ for all $t\in D$. Since $u\in C([0,+\infty),X)$, we obtain
that $u\left(  t\right)  =\gamma\left(  t+\tau\right)  $ for all $t\geq0$, so
that $\gamma\left(  \text{\textperiodcentered}+\tau\right)  \in\mathcal{R}.$
As $\tau\in\mathbb{R}$ is arbitrary, we obtain that $\gamma$ is a complete
trajectory of $\mathcal{R}$.
\end{proof}

\bigskip

Let $\mathbb{K}$ be the set of all bounded complete trajectories of
$\mathcal{R}$. Now we will establish equality (\ref{6}) in the abstract setting.

\begin{theorem}
\label{CaractAttrAbs}Assume that $\left(  K1\right)  -\left(  K2\right)
,\ \left(  K4\right)  $ hold and that $G$ possesses a compact global attractor
$\Theta.$ Then%
\begin{equation}
\Theta=\{\gamma\left(  0\right)  :\gamma\left(  \text{\textperiodcentered
}\right)  \in\mathbb{K}\}=\cup_{t\in\mathbb{R}}\{\gamma\left(  t\right)
:\gamma\left(  \text{\textperiodcentered}\right)  \in\mathbb{K}\}.
\label{EqAttrComplete}%
\end{equation}

\end{theorem}

\begin{proof}
Let $\gamma\left(  \text{\textperiodcentered}\right)  \in\mathbb{K}$. We note
that $B_{\gamma}=\cup_{s\in\mathbb{R}}\gamma\left(  s\right)  \subset G\left(
t,B_{\gamma}\right)  $ implies (as $B_{\gamma}$ is bounded) that%
\[
dist_{X}(B_{\gamma},\Theta)\leq dist_{X}(G\left(  t,B_{\gamma}\right)
,\Theta)\rightarrow0\text{ as }t\rightarrow+\infty,
\]
so that $B_{\gamma}\subset\Theta$.

Conversely, let $z\in\Theta$. Since $\Theta\subset G\left(  t,\Theta\right)
,$ we have $z\in G\left(  t_{n},\Theta\right)  $ with $t_{n}\rightarrow\infty
$. Hence, $z=u_{n}\left(  t_{n}\right)  $, where $u_{n}\in\mathcal{R}$ and
$u_{n}\left(  0\right)  \in\Theta$. Consider the functions $v_{n}^{0}\left(
\text{\textperiodcentered}\right)  =u_{n}\left(  \text{\textperiodcentered
}+t_{n}\right)  $, which belong to $\mathcal{R}$. In view of $\left(
K4\right)  $ there exist $v^{0}\left(  \text{\textperiodcentered}\right)
\in\mathcal{R}$ with $v^{0}\left(  0\right)  =z$ and a subsequence (denoted
again by $u_{n}$) such that $v_{n}^{0}\left(  t\right)  \rightarrow
v^{0}\left(  t\right)  $ in $X$ for all $t\geq0.$ Since, $v^{0}\left(
t\right)  =\lim_{n\rightarrow\infty}u^{n}\left(  t+t_{n}\right)  $, we obtain
that $v^{0}\left(  t\right)  \in\Theta$ for all $t\geq0$. Let us take a
sequence $t_{j}\rightarrow+\infty$ such that $t_{0}=0<t_{j}<t_{j+1}$ for any
$j\in\mathbb{N}$. Consider now the sequence of functions $v_{n}^{1}\left(
\text{\textperiodcentered}\right)  =u_{n}\left(  \text{\textperiodcentered
}+t_{n}-t_{1}\right)  $, which belong to $\mathcal{R}$. By (\ref{Attraction})
it is clear that (up to a subsequence) $v_{n}^{1}\left(  0\right)  $ is
convergent in $X$. As before there exist then $v^{1}\left(
\text{\textperiodcentered}\right)  \in\mathcal{R}$ and a subsequence (denoted
again by $v_{n}$) such that $v_{n}^{1}\left(  t\right)  \rightarrow
v^{1}\left(  t\right)  $ in $X$ for all $t\geq0.$ Also, $v^{1}\left(
t\right)  \in\Theta$ and $v^{1}\left(  t+t_{1}\right)  =v^{0}\left(  t\right)
$ for all $t\geq0$. In this way we can define inductively a sequence
$v^{j}\left(  \text{\textperiodcentered}\right)  \in\mathcal{R}$ such that
$v^{j}\left(  t\right)  \in\Theta$ and $v^{j}\left(  t+t_{j}-t_{j-1}\right)
=v^{j-1}\left(  t\right)  $ for all $t\geq0$ and $j\in\mathbb{N}$. We define
the function $v\left(  t\right)  $ by taking for all $t\in\mathbb{R}$ the
commom value at $t$ of the functions $v^{j}\left(  \text{\textperiodcentered
}\right)  $. Namely, for any $j$ such that $t\geq-t_{j}$ we put%
\[
v\left(  t\right)  =v^{j}\left(  t+t_{j}\right)  \text{.}%
\]
Then $v\left(  \text{\textperiodcentered}\right)  $ is a complete trajectory
of $\mathcal{R}$, $v\left(  0\right)  =z$ and $v\left(  t\right)  \in\Theta$
for all $t\in\mathbb{R}$. Hence, $v\left(  \text{\textperiodcentered}\right)
\in\mathbb{K}$.

The first equality is proved. The second one is obvious from the definition of
a complete trajectory.
\end{proof}

\bigskip

The last theorem is also true if we replace $\left(  K4\right)  $ by $\left(
K3\right)  .$

\begin{theorem}
\label{CaractAttrAbs2}Assume that $\left(  K1\right)  -\left(  K3\right)  $
hold and that $G$ possesses a compact global attractor $\Theta.$ Then%
\begin{equation}
\Theta=\{\gamma\left(  0\right)  :\gamma\left(  \text{\textperiodcentered
}\right)  \in\mathbb{K}\}=\cup_{t\in\mathbb{R}}\{\gamma\left(  t\right)
:\gamma\left(  \text{\textperiodcentered}\right)  \in\mathbb{K}\}.\nonumber
\end{equation}

\end{theorem}

\begin{proof}
As in the proof of Theorem \ref{CaractAttrAbs} we obtain that $B_{\gamma}%
=\cup_{t\in\mathbb{R}}\gamma\left(  t\right)  \subset\Theta$ for any
$\gamma\left(  \text{\textperiodcentered}\right)  \in\mathbb{K}$.

We note that Lemma \ref{MS} implies that $G$ is strict, and then by%
\[
G\left(  t,\Theta\right)  \subset G\left(  t,G\left(  \tau,\Theta\right)
\right)  =G(t+\tau,\Theta)\rightarrow\Theta\text{ as }\tau\rightarrow+\infty
\]
we have that $G\left(  t,\Theta\right)  \subset\Theta$ for any $t\geq0$, so
that $\Theta$ is strictly invariant.

We take an arbitrary $z\in\Theta$. We take $\psi^{0}\left(
\text{\textperiodcentered}\right)  \in\mathcal{R}$ such that $\psi^{0}\left(
0\right)  =z$. Since $\Theta$ is strictly invariant, we have $\psi^{0}\left(
t\right)  \in G(t,z)\subset\Theta$ for all $t\geq0.$ Let us take a sequence
$t_{j}\rightarrow+\infty$ such that $t_{0}=0<t_{j}<t_{j+1}$ for any
$j\in\mathbb{N}$. From $z\in G\left(  t_{1},\Theta\right)  $ there exists
$z_{1}\in\Theta$ and $\varphi^{1}\left(  \text{\textperiodcentered}\right)
\in\mathcal{R}$ such that $z=\varphi^{1}\left(  t_{1}\right)  $ and
$\varphi^{1}\left(  0\right)  =z_{1}$. By $\left(  K3\right)  $ we can
concatenate $\varphi^{1}$ and $\varphi^{0}$ and obtain a $\psi^{1}\left(
\text{\textperiodcentered}\right)  \in\mathcal{R}$ satisfying $\psi^{1}\left(
t_{1}\right)  =z$, $\psi^{1}\left(  t\right)  \in\Theta$, for all $t\geq0$,
and $\psi^{1}\left(  t+t_{1}\right)  =\psi^{0}\left(  t\right)  $, for all
$t\geq0$. Inductively, we can define for any $j\geq1$ a $\psi^{j}\left(
\text{\textperiodcentered}\right)  \in\mathcal{R}$ satisfying $\psi^{j}\left(
t+t_{j}-t_{j-1}\right)  =\psi^{j-1}\left(  t\right)  $ and $\psi^{j}\left(
t\right)  \in\Theta$. For any $t\in\mathbb{R}$ let $\psi\left(  t\right)  $ be
the common value of $\psi^{j}\left(  t+t_{j}\right)  $ for $t\geq-t_{j}$. Then
$\psi\left(  \text{\textperiodcentered}\right)  \in\mathbb{K}$ and
$\psi\left(  0\right)  =z$.

The second one is obvious from the definition of a complete trajectory.
\end{proof}

\bigskip

\begin{remark}
The map $\gamma:\mathbb{R}\rightarrow X$ is a complete trajectory of $G$\ if
(\ref{PropCompleteTraj}) holds. Let $\mathbb{K}_{G}$ be the set of all bounded
complete trajectories of $G$. If $\left(  K1\right)  -\left(  K4\right)  $
hold, then by Lemma \ref{CompleteTrajEquiv} we have $\mathbb{K}_{G}%
=\mathbb{K}$ and equality (\ref{EqAttrComplete}) is the same as
\begin{equation}
\Theta=\{\gamma\left(  0\right)  :\gamma\left(  \text{\textperiodcentered
}\right)  \in\mathbb{K}_{G}\}=\cup_{t\in\mathbb{R}}\{\gamma\left(  t\right)
:\gamma\left(  \text{\textperiodcentered}\right)  \in\mathbb{K}_{G}\}.
\label{EqAttrCompleteG}%
\end{equation}

If either $\left(  K3\right)  $ or $\left(  K4\right)  $ fails to be true,
then we can say only that $\mathbb{K}\subset\mathbb{K}_{G}$. Nevertheless, if
either $\left(  K3\right)  $ or $\left(  K4\right)  $ holds, then
(\ref{EqAttrCompleteG}) is still true. Indeed, Theorems \ref{CaractAttrAbs},
\ref{CaractAttrAbs2} imply that%
\[
\Theta=\{\gamma\left(  0\right)  :\gamma\left(  \text{\textperiodcentered
}\right)  \in\mathbb{K}\}\subset\{\gamma\left(  0\right)  :\gamma\left(
\text{\textperiodcentered}\right)  \in\mathbb{K}_{G}\}
\]
and as in the proof of Theorem \ref{CaractAttrAbs} we obtain that $B_{\gamma
}=\cup_{t\in\mathbb{R}}\gamma\left(  t\right)  \subset\Theta$ for any
$\gamma\left(  \text{\textperiodcentered}\right)  \in\mathbb{K}_{G}$, so that
(\ref{EqAttrCompleteG}) holds.

If both $\left(  K3\right)  $ and $\left(  K4\right)  $ fail, then equality
(\ref{EqAttrCompleteG}) can be obtained under some assumptions on $G$. Namely,
in \cite[Lemmas 2.25 and 2.27]{KMVY} it is shown that (\ref{EqAttrCompleteG})
holds if either $G$ is strict or the following condition is true: for any
sequence $\varphi^{n}:\mathbb{R}^{+}\rightarrow X$ satisfying
(\ref{PropCompleteTraj}), for $s,t\geq0,$ and $\varphi^{n}\left(  0\right)
\rightarrow\varphi_{0}$ in $X$, there exists a subsequence and $\varphi
:\mathbb{R}^{+}\rightarrow X$ satisfying (\ref{PropCompleteTraj}) for
$s,t\geq0$ such that
\[
\varphi^{n_{k}}\left(  t\right)  \rightarrow\varphi\left(  t\right)  \text{
for any }t\geq0.
\]

\end{remark}

\bigskip

We shall apply these results to the set $K^{+}$ generated by the weak
solutions of (\ref{1}). We note that in view of Lemmas 3 and 15 in
\cite{KapustValero06} (see also \cite[Theorems 3.11 and 3.18]{KMVY})
assumptions $\left(  K1\right)  -\left(  K4\right)  $ are satisfied for
$K^{+}$. Moreover, the sets of stationary points $\mathfrak{R}$ of problem
(\ref{1}) coincides with the set $\mathfrak{R}_{\mathcal{R}}=\mathfrak{R}%
_{K^{+}}$ given in Definition \ref{DefFixed}.

\begin{lemma}
\label{EqFixed}Let (\ref{2}) hold. Then $\mathfrak{R}=\mathfrak{R}_{K^{+}}.$
\end{lemma}

\begin{proof}
Let $u_{0}\in\mathfrak{R}_{K^{+}}$. Then $u\left(  t\right)  \equiv u_{0}$
belongs to $K^{+}$. Therefore, $u\left(  \text{\textperiodcentered}\right)  $
satisfies (\ref{3}), so that (\ref{Stationary}) holds. Conversely, let
$v\in\mathfrak{R}$. Then it is obvious that $v\left(  t\right)  \equiv v_{0}$
is a weak solution, so that is belongs to $K^{+}$.
\end{proof}

\bigskip

Then Lemmas \ref{FixedPointsCar}, \ref{CompleteTrajEquiv} and Theorem
\ref{CaractAttrAbs} imply the following result.

\begin{theorem}
\label{PropK+}Let (\ref{2}) hold. Then the set of weak solutions $K^{+}$ of
(\ref{1}) satisfies:

\begin{enumerate}
\item $z\in\mathfrak{R}$ if and only if $z\in G\left(  t,z\right)  $ for all
$t\geq0.$

\item The map $\gamma:\mathbb{R}\rightarrow L^{2}\left(  \Omega\right)  $ is a
complete trajectory of $K^{+}$ if and only if (\ref{PropCompleteTraj}) holds.

\item The compact global attractor $\Theta$ of $G$ satisfies (\ref{6}).
\end{enumerate}
\end{theorem}

\section{Structure of the global attractor for weak solutions\label{StrWeak}}

In this section we will study the structure of the global attractor generated
by weak solutions of equation (\ref{1}).

First, let us prove some regularity properties of the stationary points.

\begin{lemma}
\label{lem:1} Under conditions (\ref{2}) the set $\mathfrak{R}$ of solutions
of the problem
\begin{equation}
\left\{
\begin{array}
[c]{l}%
-\Delta u+f(u)=h,\quad x\in\Omega,\\
u|_{\partial\Omega}=0,
\end{array}
\right.  \label{9}%
\end{equation}
is nonempty, compact in $L^{2}(\Omega)$, and bounded in $H_{0}^{1}(\Omega)\cap
H^{2}\left(  \Omega\right)  $.
\end{lemma}

\begin{proof}
Due to $f(u)u\geq-C_{2},$ for all $u\in\mathbb{R}$, the operator $L=-\Delta
u+f(u):H_{0}^{1}(\Omega)\rightarrow H^{-1}(\Omega)$ is coercive,
$-\Delta:H_{0}^{1}(\Omega)\rightarrow H^{-1}(\Omega)$ is monotone and
continuous. Also, from (\ref{2}) we can obtain that $f:H_{0}^{1}%
(\Omega)\rightarrow H^{-1}(\Omega)$ is strongly continuous (i.e.
$u_{n}\rightarrow u$ weakly in $H_{0}^{1}\left(  \Omega\right)  $ implies
$f\left(  u_{n}\right)  \rightarrow f\left(  u\right)  $ in $H^{-1}\left(
\Omega\right)  $). Hence, $L$ is pseudomonotone, coercive and bounded, so that
a classical theorem of Brezis (see \cite{Zeidler}) implies that $\mathfrak{R}%
\neq\emptyset$. It is clear also that it is weakly compact in $H_{0}%
^{1}(\Omega)$ and therefore compact in $L^{2}(\Omega)$. We remark that
$\mathfrak{R}$ is bounded in $H_{0}^{1}(\Omega)$, as $L:H_{0}^{1}%
(\Omega)\rightarrow H^{-1}(\Omega)$ is coercive. Hence, the equality%
\[
-\Delta u+f\left(  u\right)  =h
\]
and (\ref{2}), together with the continuous imbedding $H_{0}^{1}\left(
\Omega\right)  \subset L^{6}\left(  \Omega\right)  ,$ imply
\[
\left\Vert \Delta u\right\Vert ^{2}\leq C\text{ for all }u\in\mathfrak{R}%
\text{.}%
\]
Thus, $\mathfrak{R}$ is bounded in $H_{0}^{1}(\Omega)\cap H^{2}\left(
\Omega\right)  $.
\end{proof}

\bigskip

For initial data in $H_{0}^{1}\left(  \Omega\right)  $ we shall obtain the
existence of more regular solutions for (\ref{1}).

\begin{lemma}
\label{lem:4}Assume that (\ref{2}) holds. Let $\,u_{0}\in H_{0}^{1}(\Omega).$
Then there exists at least one weak solution $u$ of (\ref{1}) such that
$u(0)=u_{0}$, $u\in C([0,+\infty);H_{0}^{1}(\Omega))$ and
\begin{equation}
\Vert u(t)\Vert_{H_{0}^{1}(\Omega)}^{2}\leq C\left(  1+\Vert u_{0}\Vert
_{H_{0}^{1}(\Omega)}^{4}\right)  ,\ \forall t\geq0, \label{13}%
\end{equation}%
\begin{equation}
\int\limits_{0}^{+\infty}\Vert u_{t}(s)\Vert^{2}ds\leq C\left(  1+\Vert
u_{0}\Vert_{H_{0}^{1}(\Omega)}^{4}\right)  , \label{14}%
\end{equation}
for some $C>0$.
\end{lemma}

\begin{proof}
We take as in \cite[p.281]{ChepVishikBook} the Galerkin approximations using
the basis of eigenfunctions $\{w_{j}\left(  x\right)  $, $j\in\mathbb{N\}}$,
of the Laplace operator with Dirichlet boundary conditions. Let $X_{m}%
=\{w_{1},...,w_{m}\}$ and let $P_{m}$ be the orthogonal projector from
$L^{2}\left(  \Omega\right)  $ onto $X_{m}$. Then $u_{m}\left(  t,x\right)
=\sum_{j=i}^{m}a_{j,m}\left(  t\right)  w_{j}\left(  x\right)  $ will be a
solution of the system of ordinary differential equations%
\begin{equation}
\frac{du_{m}}{dt}=P_{m}\Delta u_{m}-P_{m}f\left(  u_{m}\right)  +P_{m}%
h,\ u_{m}\left(  0\right)  =P_{m}u_{0}. \label{Gal}%
\end{equation}
It is proved in \cite[p.281]{ChepVishikBook} that passing to a subsequence
$u_{m}$ converges to a weak solution $u$ of (\ref{1}) weakly star in
$L^{\infty}\left(  0,T;L^{2}\left(  \Omega\right)  \right)  $, weakly in
$L^{4}\left(  0,T;L^{4}\left(  \Omega\right)  \right)  $ and weakly in
$L^{2}\left(  0,T;H_{0}^{1}\left(  \Omega\right)  \right)  $ for all $T>0$.
Also, $u_{mt}\rightarrow u_{t}$ weakly in $L^{\frac{4}{3}}\left(
0,T;H^{-s}\left(  \Omega\right)  \right)  $ for some $s>0.$

Multiplying (\ref{Gal}) by $u_{mt}$ we get%
\[
\frac{d}{dt}\left(  \Vert u_{m}\Vert_{H_{0}^{1}(\Omega)}^{2}+2(F(u_{m}%
),1)-2(h,u_{m})\right)  +2\Vert u_{mt}\Vert^{2}=0,
\]
so by (\ref{PropF}),%
\[
\Vert u_{m}(t)\Vert_{H_{0}^{1}(\Omega)}^{2}+2\int\limits_{0}^{t}\Vert
u_{mt}(s)\Vert^{2}ds
\]%
\[
\leq\Vert u_{m}\left(  0\right)  \Vert_{H_{0}^{1}(\Omega)}^{2}+R_{1}\Vert
u_{m}\left(  0\right)  \Vert_{L^{4}(\Omega)}^{4}+2\Vert h\Vert\Vert
u_{m}(t)\Vert+2\Vert h\Vert\Vert u_{0}\Vert+R_{2}.
\]
So from the Poincar\'{e} inequality we obtain%
\[
\frac{1}{2}\Vert u_{m}\left(  t\right)  \Vert_{H_{0}^{1}(\Omega)}^{2}%
+2\int\limits_{0}^{t}\Vert\frac{d}{ds}u_{m}(s)\Vert^{2}ds\leq R_{3}\Vert
u_{m}\left(  0\right)  \Vert_{H_{0}^{1}(\Omega)}^{4}+R_{4},
\]
where $R_{j}>0$.

By the choise of the special basis we have that $u_{m}\left(  0\right)
\rightarrow u_{0}$ in $H_{0}^{1}\left(  \Omega\right)  $. Then we have
\[
\int\limits_{0}^{t}\Vert\frac{d}{ds}u(s)\Vert^{2}ds\leq\liminf
\limits_{m\rightarrow\infty}\int\limits_{0}^{t}\Vert\frac{d}{ds}u_{m}%
(s)\Vert^{2}ds\leq R_{5}\left(  \Vert u_{0}\Vert_{H_{0}^{1}(\Omega)}%
^{4}+1\right)  ,
\]
so that (\ref{14}) holds and $u_{mt}\rightarrow u_{t}$ weakly in $L^{2}\left(
0,T;L^{2}\left(  \Omega\right)  \right)  $. Thus from the Ascoli-Arzel\`{a}
theorem $\{u_{m}\}$ is pre-compact in $C([0,T];L^{2}(\Omega))$ and then
$u_{m}\rightarrow u$ in $C([0,T];L^{2}(\Omega)).$

On the other hand, for any $t\geq0$ up to a subsequence $u_{m_{n}%
}(t)\rightarrow a$ weakly in $H_{0}^{1}(\Omega)$. But $u_{m_{n}}(t)\rightarrow
u(t)$ in $L^{2}(\Omega)$, so that $a=u(t)$ and
\[
\Vert u(t)\Vert_{H_{0}^{1}(\Omega)}^{2}\leq\liminf\ \Vert u_{m_{n}}%
(t)\Vert_{H_{0}^{1}(\Omega)}^{2}\leq R_{5}\left(  \Vert u_{0}\Vert_{H_{0}%
^{1}(\Omega)}^{4}+1\right)  ,
\]
so that (\ref{13}) holds.

As $u\in L^{\infty}(0,T;H_{0}^{1}(\Omega))\bigcap C([0,T];L^{2}(\Omega))$, we
have $u\in C([0,T];H_{0w}^{1}(\Omega))$, where $H_{0w}^{1}(\Omega)$ is the
space $H_{0}^{1}\left(  \Omega\right)  $ with the weak topology. Moreover, the
equality $\Delta u=u_{t}+f\left(  u\right)  -h$ and (\ref{2}), (\ref{13}),
(\ref{14}) imply that $u\in L_{loc}^{2}\left(  0,+\infty;D\left(  A\right)
\right)  $. Thus, by standard results \cite[p.102]{SellYou}, we obtain that
$u\in C([0,+\infty),H_{0}^{1}\left(  \Omega\right)  ).$
\end{proof}

\bigskip

Now we are ready to prove the main result of this section about the structure
of the global attractor. From now on for any $A\subset L^{2}\left(
\Omega\right)  $ we will denote by $\overline{A}$ its closure in $L^{2}\left(
\Omega\right)  .$

\begin{theorem}
\label{teor2} Under conditions (\ref{2}) for the global attractor $\Theta$ it
holds%
\begin{equation}
\Theta=\overline{M^{-}(\mathfrak{R})}. \label{17}%
\end{equation}
If, additionally, for any $u_{0}\in H_{0}^{1}(\Omega)$ problem (\ref{1}) has a
unique weak solution with $u(0)=u_{0}$, then $\Theta$ is bounded in $H_{0}%
^{1}(\Omega)$ and
\begin{equation}
\Theta=M^{+}(\mathfrak{R})=M^{-}(\mathfrak{R}). \label{18}%
\end{equation}

\end{theorem}

\begin{proof}
First of all $\overline{\Theta\bigcap H_{0}^{1}(\Omega)}=\Theta,$ as for any
$\gamma\in\mathbb{K}$ we have that $\gamma\left(  t\right)  \in H_{0}%
^{1}(\Omega)$ for a.a. $t\in\mathbb{R}$ and $\gamma:\mathbb{R}\rightarrow
L^{2}(\Omega)$ is continuous. Let us prove that $\Theta\bigcap H_{0}%
^{1}(\Omega)={M^{-}(\mathfrak{R})}$. Let $z\in\Theta\bigcap H_{0}^{1}(\Omega
)$. Due to Theorem \ref{teor1} there exists $\gamma\in\mathbb{K}$ such
that$\,\ \gamma(0)=z$. Due to Lemma \ref{lem:4} there exist a weak solution
$u(\cdot)$ of (\ref{1}) satisfying (\ref{13}), (\ref{14}) and $\ u(0)=z$. Then
$\left(  K3\right)  $ implies that%
\[
\tilde{\gamma}(t)=\left\{
\begin{array}
[c]{l}%
\gamma(t),\,\,\ t<0\\
u(t),\,\,\ t\geq0
\end{array}
\right.  ,\,\ \tilde{\gamma}(0)=z,
\]
belongs to $\mathbb{K}$.

Let us prove that $\mathrm{dist}_{L^{2}(\Omega)}(u(t),\mathfrak{R}%
)\rightarrow0,\,\ $as$\ t\rightarrow+\infty$.

Let us take arbitrary $t_{n}\rightarrow\infty$. From (\ref{14}),
\[
\,\,\ \int\limits_{t_{n}-T}^{t_{n}}\Vert u_{t}(s)\Vert^{2}ds\rightarrow
0,\,\text{as}\,\ n\rightarrow\infty\text{, }\forall T>0.
\]
So there exists $t_{n}^{\prime}\in\lbrack t_{n}-T,t_{n}]$ such that
\[
\Vert u_{t}(t_{n}^{\prime})\Vert\rightarrow0,\,\,\,\ n\rightarrow\infty.
\]
From (\ref{13}) up to a subsequence $u(t_{n}^{\prime})\rightarrow\tilde{u}$
weakly in $H_{0}^{1}(\Omega)$. Then $u(t_{n}^{\prime})\rightarrow\tilde{u}$ in
$L^{2}(\Omega)$, so that $u(t_{n}^{\prime},x)\rightarrow\tilde{u}(x)$ a.e.,
and from \cite[p.12, Lemma 1.3]{Lions} we have $f(u(t_{n}^{\prime
}))\rightarrow f(\tilde{u})$ weakly in $L^{2}(\Omega)$.

The following equality
\[
u_{t}(t)=\Delta u(t)-f(u(t))-h
\]
in $H^{-1}(\Omega)$ is true for a.a. $t$. We can take $t_{n}^{\prime}$ from
this set of full measure and then we have%
\[
\Delta u(t_{n}^{\prime})-h=f(u(t_{n}^{\prime}))+u_{t}(t_{n}^{\prime
})\rightarrow f(\tilde{u})\text{ }\ \text{weakly in }L^{2}\left(
\Omega\right)  \text{ and then in }H^{-1}(\Omega).
\]
From this $u(t_{n}^{\prime})\rightarrow\tilde{u}$ in $H_{0}^{1}(\Omega)$ and
\[
\Delta\tilde{u}-f(\tilde{u})=h\mbox { in }H^{-1}(\Omega),\text{ that is,
}\tilde{u}\in\mathfrak{R}.
\]

Let us show that up to a subsequence $\mathrm{dist}_{L^{2}(\Omega)}%
(u(t_{n}),\mathfrak{R})\rightarrow0,\,$as$\,\ n\rightarrow\infty$. From
(\ref{13}) $u(t_{n})\rightarrow a$ weakly in $H_{0}^{1}(\Omega)$. Also%
\[
u(t_{n})\in G(t_{n}-t_{n}^{\prime},u(t_{n}^{\prime}))\,\
\]
implies%
\[
\ u(t_{n})=\varphi_{n}(t_{n}-t_{n}^{\prime}),\,\ \varphi_{n}(0)=u(t_{n}%
^{\prime}),\,\ \varphi_{n}\in K^{+}.
\]
As $u(t_{n}^{\prime})\rightarrow\tilde{u}$ in $L^{2}\left(  \Omega\right)  $,
$t_{n}-t_{n}^{\prime}\rightarrow\tau\in\lbrack0,T]$, by Theorem 3.11 in
\cite{KMVY} (see also \cite[Lemma 2]{KapustValero06}) passing to a subsequence
we have%
\[
\varphi_{n}(t_{n}-t_{n}^{\prime})\rightarrow\varphi(\tau)\in G(\tau,\tilde
{u}),
\]
that is,
\[
a\in G(\tau,\tilde{u}).
\]
We take in the previous arguments $T_{k}\downarrow0$. Then for any $k\geq1$
there exist $\tilde{u}_{k}\in\mathfrak{R}$, $\tau_{k}\in\lbrack0,T_{k}]$ such
that%
\[
a\in G(\tau_{k},\tilde{u}_{k}).
\]
Since $\mathfrak{R}$ is compact in $L^{2}(\Omega)$, up to a subsequence
$\tilde{u}_{k}\rightarrow\tilde{u}\in\mathfrak{R}$, as $\tau_{k}\rightarrow0$.
Thus by Theorem 3.11 in \cite{KMVY} we have $a\in G(0,\tilde{u})$ and then
$a=\tilde{u}$. By Theorem \ref{PropK+} $a\in\mathfrak{R}$.

So, from this we easy deduce that
\[
\mathrm{dist}_{L^{2}(\Omega)}(u(t),\mathfrak{R})\rightarrow
0,\,\,\ t\rightarrow+\infty.
\]
Hence, (\ref{17}) is proved.

Now let (\ref{1}) have a unique weak solution for every $u_{0}\in H_{0}%
^{1}(\Omega)$ (for example, it is true if $\left(  f(u)-f(v)\right)
(u-v)\geq-C|u-v|^{2},$ $\forall u,v\in\mathbb{R}$, for some $C>0$).

Then for any $z\in\Theta\,\ $we have $z=\gamma(0)=G(\tau,\gamma(-\tau))$ and
$\gamma(\tau)\in H_{0}^{1}(\Omega)$ for a.a. $\tau$. So $\gamma(t)=G(t+\tau
,\gamma(-\tau))$, $\forall t\geq0\,$, and if we repeat for the point
$\gamma(-\tau)\in H_{0}^{1}(\Omega)$ all the previous arguments, we obtain
$z=\gamma(0)\in H_{0}^{1}(\Omega)$ (by Lemma \ref{lem:4}) and $\mathrm{dist}%
_{L^{2}(\Omega)}(\gamma(t),\mathfrak{R})\rightarrow0,\,\,\ t\rightarrow
+\infty$. Then $\Theta\subset H_{0}^{1}(\Omega)$ and $\Theta=M^{-}%
(\mathfrak{R})$.

Moreover, $\Theta$ is bounded in $H_{0}^{1}(\Omega)$. Indeed, for $z\in
\Theta,\,\ z=\gamma(0)\in H_{0}^{1}(\Omega)$ and from (\ref{13}) and the
uniqueness of the solution we get
\[
\Vert z\Vert_{H_{0}^{1}(\Omega)}^{2}\leq C\left(  1+\Vert\gamma(\tau
)\Vert_{H_{0}^{1}(\Omega)}^{4}\right)  \,\,\,\ \forall\tau\leq0.
\]
For every $t\geq\tau$, by standard estimates from (\ref{4}), $\gamma(\cdot)$
satisfies
\[
\int\limits_{\tau}^{t}\Vert\gamma(s)\Vert_{H_{0}^{1}(\Omega)}^{2}ds\leq
\Vert\gamma(\tau)\Vert^{2}+\tilde{C}(t-\tau).
\]
Since $\Theta$ is bounded in $L^{2}(\Omega)$, for some $\tau^{\prime}%
\in(-1,0)$ we have $\Vert\gamma(\tau^{\prime})\Vert_{H_{0}^{1}(\Omega)}%
\leq\tilde{K}$, where $\tilde{K}$ does not depend on $\gamma$. So, $\Vert
z\Vert_{H_{0}^{1}(\Omega)}^{2}\leq C(1+\tilde{K}^{4}).$

Let us prove $\Theta=M^{+}(\mathfrak{R})$. Let $z\in\Theta$. Then
$z=\gamma(0)$, $\gamma\in\mathbb{K}$, and from the uniqueness of the solution
and Lemma \ref{lem:4} we have
\begin{equation}%
\begin{array}
[c]{l}%
\Vert\gamma(t)\Vert_{H_{0}^{1}(\Omega)}^{2}\leq C\left(  1+\Vert\gamma
(\tau)\Vert_{H_{0}^{1}(\Omega)}^{4}\right)  ,\\
\int\limits_{\tau}^{t}\Vert\gamma_{t}(s)\Vert^{2}ds\leq C\left(  1+\Vert
\gamma(\tau)\Vert_{H_{0}^{1}(\Omega)}^{4}\right)  ,\text{ }\forall t\geq\tau.
\end{array}
\label{19}%
\end{equation}
$\Theta$ is bounded in $H_{0}^{1}(\Omega)$, so from (\ref{19}) there exists
$K>0$ such that%
\begin{equation}%
\begin{array}
[c]{l}%
\Vert\gamma(t)\Vert_{H_{0}^{1}(\Omega)}\leq K\,\,\ \forall t\leq0,\\
\int\limits_{-\infty}^{0}\Vert\gamma_{t}(s)\Vert^{2}ds\leq K.
\end{array}
\label{20}%
\end{equation}
After that we can repeat the previous arguments on $(-\infty,0)$ and obtain
that
\[
\mathrm{dist}_{L^{2}(\Omega)}(\gamma(t),\mathfrak{R})\rightarrow
0,\,\ t\rightarrow-\infty.
\]

The theorem is proved.
\end{proof}

\begin{remark}
\label{rem:2} Even in the case of uniqueness we cannot use the Lyapunov
function method as in \textbf{\cite{Temam2}}, because we know nothing about
the boundedness of $\Theta$ in $H^{2}(\Omega)\bigcap H_{0}^{1}(\Omega
)$\textbf{. }However, it is possible to use the Lyapunov function if the
attractor is compact in $H_{0}^{1}(\Omega)$, as we will see in the next section.
\end{remark}

\begin{remark}
\label{rem:angl} Under condition (\ref{2}) we also have $\overline
{M^{+}(\mathfrak{R})}\subset\Theta$.
\end{remark}

\bigskip

When $\mathfrak{R}$ is finite, we can write $M^{-}(\mathfrak{R})$
($M^{+}(\mathfrak{R})$) as the union of the corresponding sets for each of the
stationary points. For $z\in\mathfrak{R}$ let
\[%
\begin{array}
[c]{c}%
M^{-}(z)=\left\{  y:\,\exists\gamma(\cdot)\in\mathbb{K},\,\ \gamma
(0)=y,\,\,\,\ \left\Vert \gamma(t)-z\right\Vert \rightarrow0,\,\ t\rightarrow
+\infty\right\}  ,\\
M^{+}(z)=\left\{  y:\,\exists\gamma(\cdot)\in\mathbb{K},\,\ \gamma
(0)=y,\,\,\,\ \left\Vert \gamma(t)-z\right\Vert \rightarrow0,\,\ t\rightarrow
-\infty\right\}  .
\end{array}
\]

\begin{lemma}
\label{lem:2}Let (\ref{2}) hold. If $\mathfrak{R}=\{z_{i}\}_{i=1}^{n}$, then
$M^{\pm}(\mathfrak{R})=\bigcup\limits_{i=1}^{n}M^{\pm}(z_{i})$.
\end{lemma}

\begin{proof}
Let $y\in M^{+}(\mathfrak{R})$. Then there exists $\gamma\in\mathbb{K}%
\,\ $such that$\ \gamma(0)=y,\,\ \mathrm{dist}_{L^{2}(\Omega)}(\gamma
(t),\mathfrak{R})\rightarrow0,\,$as$\ t\rightarrow-\infty$. For any $\tau<0$
the set $\Gamma_{\tau}=\overline{\bigcup\limits_{t\leq\tau}\gamma(t)}$ is
connected and compact (as $\Gamma_{\tau}\subset\Theta$). So, $\bigcap
\limits_{\tau<0}\Gamma_{\tau}$ is connected and compact. As for all
$\varepsilon>0\,$there exists $T<0$ such that $\gamma(t)\in\mathcal{O}%
_{\varepsilon}(\mathfrak{R})$, $\forall t\leq T$, we have $\bigcap
\limits_{\tau<0}\Gamma_{\tau}\subset\mathfrak{R}\,$. Then%
\[
\ \bigcap\limits_{\tau<0}\Gamma_{\tau}=\{z_{i_{0}}\}\subset\mathfrak{R}\,,\,\
\]
so that%
\[
\gamma(t)\rightarrow z_{i_{0}},\,\text{as}\ t\rightarrow-\infty\,\ \text{and
}\ y\in M^{+}(z_{i_{0}}).
\]

For $M^{-}(\mathfrak{R})$ the proof is similar.
\end{proof}

\bigskip

We finish this section with a reularity result of the global attractor in the
space $L^{\infty}(\Omega)$.

\begin{lemma}
\label{lem:5} Under conditions (\ref{2}) and $h\in L^{\infty}(\Omega)$ the set
$\Theta$ is bounded in $L^{\infty}(\Omega)$.
\end{lemma}

\begin{proof}
In fact, the arguments are the same as in \cite[p.321]{Temam2}.

Let $\varphi_{+}=\max\{\varphi,0\}$. It is known that for any $u\in
\mathcal{D}(\tau,T;H_{0}^{1}(\Omega)),$ $\eta\in C_{0}^{\infty}(\tau,T)$,%
\begin{equation}
\int\limits_{\tau}^{T}(u_{t},u^{+})\eta dt=-\frac{1}{2}\int\limits_{\tau}%
^{T}\Vert u^{+}\Vert^{2}\eta_{t}dt. \label{21}%
\end{equation}
For an arbitrary complete trajectory of (\ref{1}) we have $u\in L^{2}%
(\tau,T;H_{0}^{1}(\Omega))\bigcap L^{4}(\tau,T;L^{4}(\Omega))$, $u_{t}\in
L^{2}(\tau,T;H^{-1}(\Omega))+L^{\frac{4}{3}}(\tau,T;L^{\frac{4}{3}}(\Omega))$.

So, by standard regularization we find functions $u_{n}\in\mathcal{D}%
(\tau,T;H_{0}^{1}(\Omega))$ such that
\[
u_{n}\rightarrow u\mbox { in }L^{2}(\tau,T;H_{0}^{1}(\Omega))\bigcap
L^{4}(\tau,T;L^{4}(\Omega))
\]%
\[
u_{n_{t}}\rightarrow u_{t}\mbox { in }L^{2}(\tau,T;H^{-1}(\Omega))+L^{\frac
{4}{3}}(\tau,T;L^{\frac{4}{3}}(\Omega)).
\]
As $u_{n}^{+}\rightarrow u^{+}$ in $L^{2}(\tau,T;H_{0}^{1}(\Omega))\bigcap
L^{4}(\tau,T;L^{4}(\Omega))$, we can pass to the limit in (\ref{21}) and
obtain that (\ref{21}) is true for every solution of (\ref{1}) on $[\tau,T]$.
Then putting $g=f-h$ for any $M>0$ we have
\[
\frac{1}{2}\frac{d}{dt}\Vert(u-M)^{+}\Vert^{2}+\Vert(u-M)^{+}\Vert_{H_{0}%
^{1}(\Omega)}^{2}+\int\limits_{\Omega}g(x,u)(u-M)^{+}dx=0.
\]
From (\ref{2}) and $h\in L^{\infty}(\Omega)$ for a.a. $x\in\Omega$ and
$u\in\mathbb{R},$%
\[
\tilde{\alpha}|u|^{4}-\tilde{C}_{2}\leq g(u)u\leq\tilde{C}_{1}|u|^{4}%
+\tilde{C}_{1},
\]
where $\tilde{\alpha}$ does not depend on $u,\,x$.

If $u \le M$, then $g(u)(u-M)^{+}=0$.

If $u>M$, then
\begin{align*}
g(x,u)(u-M)^{+}  &  =g(x,u)u\frac{(u-M)^{+}}{u}=g(x,u)u(1-\frac{M}{u})\\
&  \geq(\tilde{\alpha}u^{4}-\tilde{C}_{2})(1-\frac{M}{u})\geq(\tilde{\alpha
}M^{4}-\tilde{C}_{2})(1-\frac{M}{u})
\end{align*}
and if we choose $M=\left(  \frac{\tilde{C}_{2}}{\tilde{\alpha}}\right)
^{\frac{1}{4}}$, then $g(x,u)(u-M)^{+}\geq0$ a.e.

Then
\[
\frac{d}{dt}\Vert(u-M)^{+}\Vert^{2}+2\Vert(u-M)^{+}\Vert_{H_{0}^{1}(\Omega
)}^{2}\leq0
\]
and for all $t>\tau$,%
\begin{equation}
\Vert(u-M)^{+}(t)\Vert^{2}\leq\Vert(u-M)^{+}(\tau)\Vert^{2}e^{-2\lambda
_{1}(t-\tau)}. \label{22}%
\end{equation}
If $u(\cdot)\in\mathbb{K}$, then from (\ref{22}) taking $\tau\rightarrow
-\infty$ we obtain $u(x,t)\leq M,\ \forall t\in\mathbb{R}$, for a.a.
$x\in\Omega.$

In the same way we will have $u(x,t)\ge M$ (using $(u-M)^{-}$).

Then
\[
ess\sup\limits_{x\in\Omega}|z(x)|\leq M\text{, }\forall z\in\Theta.
\]

\end{proof}

\begin{remark}
\label{rem:3}The set $M^{+}\left(  \mathfrak{R}\right)  $ can be used in order
to study properties of the global attractor as the fractal dimension. Let us
consider an example which shows that a finite estimate of \ the fractal
dimension of the global attractor for problem (\ref{1}) is not preserved under
small, but unregular perturbations (even in the single-valued case).

Let $h(x)\equiv0$, $f_{k}(u)=u^{3}-k^{-\frac{1}{2}}\sin\left(  k\cdot
u\right)  $. Then for any $k\in\mathbb{Z}$ $\ f_{k}$ satisfies (\ref{2}) with
constants which do not depend on $k$. In this case $z=0\in\mathfrak{R}$,
$G_{k}(t,u_{0})=S_{k}(t)u_{0}$ is a single-valued semigroup and due to
\cite[p.496]{Temam2} $z=0$ is a hyperbolic point if $\lambda_{i}\not =%
k^{\frac{1}{2}}$ ($\lambda_{i}$ are the eigenvalues of $-\Delta$),
$M^{+}(0)\subset\Theta$ and $M^{+}(0)$ is a smooth manifold with dimension
$n_{k}$, where $n_{k}$ is the number of eigenvalues of $S_{k}^{\prime}(t)$ in
$\{|\lambda|<1\}$, that is, the number of eigenvalues of $-\Delta$ which
satisfy the inequality $\lambda_{i}<k^{\frac{1}{2}}$.

Thus, if $k\rightarrow\infty$, then for the attractors $\Theta_{k}$ we obtain
\[
\dim\Theta_{k}\geq\dim M^{+}(0)=n_{k}\rightarrow\infty,\,\,\ k\rightarrow
\infty.
\]
So, under conditions (\ref{2}), we can have arbitrary large dimension of the
global attractor, although $f_{k}\left(  u\right)  $ is a small perturbation
of $f_{0}\left(  u\right)  =u^{3}$, for which it is easy to see that
$\mathfrak{R\equiv\{}0\mathfrak{\}}$, so
\[
\dim\Theta_{0}=0.
\]

\end{remark}

\section{Existence and structure of the global attractor for regular
solutions\label{StrReg}}

We shall prove in this section that the equality
\[
\Theta=M^{-}\left(  \mathfrak{R}\right)  =M^{+}(\mathfrak{R})
\]
holds if we consider more regular solutions than in the previous section.

The function $u\in L_{loc}^{2}(0,+\infty;H_{0}^{1}(\Omega))\bigcap L_{loc}%
^{4}(0,+\infty;L^{4}(\Omega))$ is called a regular solution of (\ref{1}) on
$(0,+\infty)$ if for all $T>0,\,v\in H_{0}^{1}(\Omega)\,\ $and $\eta\in
C_{0}^{\infty}(0,T)$ we have%
\begin{equation}
-\int\limits_{0}^{T}(u,v)\eta_{t}dt+\int\limits_{0}^{T}\left(  (u,v)_{H_{0}%
^{1}(\Omega)}+(f(u),v)-(h,v)\right)  \eta dt=0, \label{Eq1}%
\end{equation}
and%
\begin{align}
u  &  \in L^{\infty}\left(  \varepsilon,T;H_{0}^{1}\left(  \Omega\right)
\right)  ,\label{LInfH1}\\
u_{t}  &  \in L^{2}\left(  \varepsilon,T;L^{2}\left(  \Omega\right)  \right)
,\ \forall\text{ }0<\varepsilon<T. \label{DerivL2}%
\end{align}

On the other hand, from (\ref{2}) we get
\[
\int_{r}^{T}\int_{\Omega}\left\vert f\left(  u\right)  \right\vert
^{2}dxdt\leq K\int_{r}^{T}\left(  1+\left\Vert u\left(  t\right)  \right\Vert
_{H_{0}^{1}\left(  \Omega\right)  }^{6}\right)  dt.
\]
Then the equality $\Delta u=u_{t}+f\left(  u\right)  -h$ and (\ref{LInfH1}%
)-(\ref{DerivL2}) imply that%
\begin{equation}
u\in L^{2}\left(  \varepsilon,T;D\left(  A\right)  \right)  \label{D(A)}%
\end{equation}
for any regular solution $u$.

\begin{theorem}
\label{ExistenceRegSol}Let (\ref{2}) hold. For any $u_{0}\in L^{2}\left(
\Omega\right)  $ there exists at least one regular solution of (\ref{1}) such
that $u\left(  0\right)  =u_{0}.$ Moreover, there exist $R_{i}>0$ such that
every regular solution with $u\left(  0\right)  =u_{0}\in L^{2}\left(
\Omega\right)  $ satisfies
\begin{equation}
\left\Vert u\left(  t+r\right)  \right\Vert _{H_{0}^{1}\left(  \Omega\right)
}^{2}\leq R_{1}\left(  \frac{e^{-\lambda_{1}t}\left\Vert u_{0}\right\Vert
^{2}+1}{r}+1+r\right)  e^{r}\text{,} \label{PropReg1}%
\end{equation}%
\begin{equation}
\left\Vert u\left(  t\right)  \right\Vert ^{2}\leq e^{-\lambda_{1}t}\left\Vert
u_{0}\right\Vert ^{2}+R_{2}, \label{PropReg2}%
\end{equation}%
\begin{equation}
\int_{r}^{T}\left\Vert u_{t}\right\Vert ^{2}dt\leq R_{3}\left(  \frac
{\left\Vert u_{0}\right\Vert ^{2}+1}{r}+1+r\right)  e^{r}, \label{PropReg3}%
\end{equation}%
\begin{equation}
\int_{r}^{T}\left\Vert \Delta u\right\Vert ^{2}dt\leq R_{4}\left(
T-r+1\right)  \left(  \frac{\left\Vert u_{0}\right\Vert ^{2}+1}{r}+1+r\right)
^{3}e^{3r}, \label{PropReg4}%
\end{equation}
for all $0\leq t<t+r<+\infty.$ Thus,
\begin{equation}
u\in C\left(  (0,+\infty),H_{0}^{1}\left(  \Omega\right)  \right)  ,
\label{PropReg5}%
\end{equation}%
\begin{equation}
\frac{d}{dt}\left\Vert u\right\Vert _{H_{0}^{1}\left(  \Omega\right)  }%
^{2}=2\left(  -\Delta u,u_{t}\right)  \text{ for a.a. }t>0. \label{PropReg6}%
\end{equation}
Moreover, the following energy equality holds%
\begin{equation}
E\left(  u\left(  t\right)  \right)  +2\int_{s}^{t}\left\Vert u_{r}\right\Vert
^{2}dr=E\left(  u\left(  s\right)  \right)  \text{, for all }t\geq s>0,
\label{Energy}%
\end{equation}
where $E\left(  u\left(  t\right)  \right)  =\left\Vert u\left(  t\right)
\right\Vert _{H_{0}^{1}\left(  \Omega\right)  }^{2}+2\left(  F\left(  u\left(
t\right)  \right)  ,1\right)  -2\left(  h,u\left(  t\right)  \right)  .$
\end{theorem}

\begin{proof}
Let $v_{0}\in H_{0}^{1}\left(  \Omega\right)  $ be arbitrary. Then by Lemma
\ref{lem:4} there exists a solution $v\left(  \text{\textperiodcentered
}\right)  \in C\left(  [0,+\infty),H_{0}^{1}\left(  \Omega\right)  \right)  $
such that%
\begin{equation}
\Vert v(t)\Vert_{H_{0}^{1}(\Omega)}^{2}\leq C\left(  1+\Vert v_{0}\Vert
_{H_{0}^{1}(\Omega)}^{4}\right)  ,\forall t\geq0, \label{Acot1}%
\end{equation}%
\begin{equation}
\int\limits_{0}^{+\infty}\Vert v_{t}(s)\Vert^{2}ds\leq C\left(  1+\Vert
v_{0}\Vert_{H_{0}^{1}(\Omega)}^{4}\right)  . \label{Acot2}%
\end{equation}
It follows by (\ref{2}) that%
\begin{align}
\left\Vert f\left(  v\left(  t\right)  \right)  \right\Vert ^{2}dt  &
\leq\int_{\Omega}C_{1}\left(  1+\left\vert u\left(  t,x\right)  \right\vert
^{3}\right)  ^{2}dx\label{Acotf}\\
&  \leq K_{1}\left(  1+\left\Vert v\left(  t\right)  \right\Vert _{H_{0}%
^{1}\left(  \Omega\right)  }^{6}\right)  \leq K_{2}\left(  1+\Vert v_{0}%
\Vert_{H_{0}^{1}(\Omega)}^{12}\right)  .\nonumber
\end{align}
Hence, the equality $\Delta v=v_{t}+f\left(  v\right)  -h$ implies that $v\in
L_{loc}^{2}\left(  0,+\infty;D\left(  A\right)  \right)  $. Thus, by standard
results \cite[p.102]{SellYou}, we obtain that
\begin{equation}
\frac{d}{dt}\left\Vert v\right\Vert _{H_{0}^{1}\left(  \Omega\right)  }%
^{2}=2\left(  -\Delta v,v_{t}\right)  \text{ for a.a. }t>0. \label{DerivH1}%
\end{equation}
Also, it is not difficult to show by regularization that $\left(  F\left(
v\left(  t\right)  \right)  ,1\right)  $ is absolutely continuous and
\begin{equation}
\frac{d}{dt}\left(  F\left(  v\left(  t\right)  \right)  ,1\right)  =\left(
v_{t},f\left(  v\left(  t\right)  \right)  \right)  \text{ for a.a. }t>0.
\label{DerivF}%
\end{equation}

Let $u_{0}^{n}\in H_{0}^{1}\left(  \Omega\right)  $ be a sequence such that
$u_{0}^{n}\rightarrow u_{0}$ in $L^{2}\left(  \Omega\right)  $ and let
$u^{n}\left(  \text{\textperiodcentered}\right)  $ be a solution of (\ref{1})
with $u^{n}\left(  0\right)  =u_{0}^{n}$ satisfying (\ref{Acot1}%
)-(\ref{Acot2}). We multiply (\ref{1}) by $u_{t}^{n}$ and using (\ref{DerivH1}%
), (\ref{DerivF}) we obtain%
\begin{equation}
2\left\Vert u_{t}^{n}\right\Vert ^{2}+\frac{d}{dt}\left(  \left\Vert
u^{n}\right\Vert _{H_{0}^{1}\left(  \Omega\right)  }^{2}+2\left(  F\left(
u^{n}\right)  ,1\right)  -2\left(  h,u^{n}\right)  \right)  =0\text{ for a.a.
}t\in\left(  0,T\right)  . \label{Inequt}%
\end{equation}
On the other hand, multiplying (\ref{1}) by $u^{n}$ and using (\ref{2}) it is
standard to obtain that $u^{n}$ satisfy%
\begin{equation}
\frac{d}{dt}\left\Vert u^{n}\right\Vert ^{2}+\lambda_{1}\left\Vert
u^{n}\right\Vert ^{2}+\left\Vert u^{n}\right\Vert _{H_{0}^{1}\left(
\Omega\right)  }^{2}+\alpha\left\Vert u^{n}\right\Vert _{L^{4}\left(
\Omega\right)  }^{4}\leq K_{3}+\left\Vert h\right\Vert ^{2}. \label{IneqUn}%
\end{equation}
By Gronwall's lemma we obtain%
\begin{equation}
\left\Vert u^{n}\left(  t\right)  \right\Vert ^{2}\leq e^{-\lambda_{1}%
t}\left\Vert u_{0}^{n}\right\Vert ^{2}+\frac{1}{\lambda_{1}}\left(
K_{3}+\left\Vert h\right\Vert ^{2}\right)  . \label{IneqUn2}%
\end{equation}
Thus integrating (\ref{IneqUn}) over $\left(  t,t+r\right)  $ with $r>0$ we
have%
\begin{align}
&  \left\Vert u^{n}\left(  t+r\right)  \right\Vert ^{2}+\int_{t}%
^{t+r}\left\Vert u^{n}\right\Vert _{H_{0}^{1}\left(  \Omega\right)  }%
^{2}ds+\alpha\int_{t}^{t+r}\left\Vert u^{n}\right\Vert _{L^{4}\left(
\Omega\right)  }^{4}ds\label{IneqUnt+r}\\
&  \leq\left\Vert u^{n}\left(  t\right)  \right\Vert ^{2}+r\left(
K_{3}+\left\Vert h\right\Vert ^{2}\right) \nonumber\\
&  \leq e^{-\lambda_{1}t}\left\Vert u_{0}^{n}\right\Vert ^{2}+\left(  \frac
{1}{\lambda_{1}}+r\right)  \left(  K_{3}+\left\Vert h\right\Vert ^{2}\right)
.\nonumber
\end{align}
Then by (\ref{PropF}),%
\begin{align}
&  \int_{t}^{t+r}\left(  \left\Vert u^{n}\right\Vert _{H_{0}^{1}\left(
\Omega\right)  }^{2}+2\left(  F\left(  u^{n}\left(  s\right)  \right)
,1\right)  -2\left(  h,u^{n}\right)  \right)  ds\nonumber\\
&  \leq\int_{t}^{t+r}\left\Vert u^{n}\right\Vert _{H_{0}^{1}\left(
\Omega\right)  }^{2}ds+K_{4}\int_{t}^{t+r}\int_{\Omega}\left(  1+\left\vert
u^{n}\right\vert ^{4}\right)  dxds+r\left\Vert h\right\Vert ^{2}+\int%
_{t}^{t+r}\left\Vert u^{n}\right\Vert ^{2}ds\nonumber\\
&  \leq K_{5}\left(  e^{-\lambda_{1}t}\left\Vert u_{0}^{n}\right\Vert
^{2}+\left(  \frac{1}{\lambda_{1}}+r\right)  \left(  1+\left\Vert h\right\Vert
^{2}\right)  \right) \nonumber\\
&  \leq K_{6}(e^{-\lambda_{1}t}\left\Vert u_{0}^{n}\right\Vert ^{2}+r+1),
\label{IneqUGL1}%
\end{align}
for all $n$ and $t\geq0.$ Also, by (\ref{Inequt}), (\ref{PropF}) and
\[
-2\left(  h,u^{n}\right)  \geq-\frac{4}{\lambda_{1}}\left\Vert h\right\Vert
^{2}-\lambda_{1}\left\Vert u^{n}\right\Vert ^{2}%
\]
we obtain
\begin{align}
&  \frac{d}{dt}\left(  \left\Vert u^{n}\right\Vert _{H_{0}^{1}}^{2}+2\left(
F\left(  u^{n}\right)  ,1\right)  -2\left(  h,u^{n}\right)  \right)
\nonumber\\
&  \leq\left\Vert u^{n}\right\Vert _{H_{0}^{1}}^{2}-\lambda_{1}\left\Vert
u^{n}\right\Vert ^{2}\nonumber\\
&  \leq\left\Vert u^{n}\right\Vert _{H_{0}^{1}}^{2}+2\left(  F\left(
u^{n}\right)  ,1\right)  -2\left(  h,u^{n}\right)  +2\widetilde{D}_{2}%
+\frac{4}{\lambda_{1}}\left\Vert h\right\Vert ^{2}, \label{IneqUGL2}%
\end{align}
where $\widetilde{D}_{2}=\int_{\Omega}D_{2}dx.$

Recall the well known uniform Gronwall lemma \cite{Temam2}.

\begin{lemma}
\label{UGL}Let $g,w,y$ be three positive integrable functions on $\left(
t_{0},+\infty\right)  $ such that $y^{\prime}$ is locally integrable on
$\left(  t_{0},+\infty\right)  $ and such that
\[
\frac{dy}{dt}\leq gy+w\text{ if }t\geq t_{0},
\]%
\[
\int_{t}^{t+r}gds\leq a_{1},\ \int_{t}^{t+r}wds\leq a_{2},\ \int_{t}%
^{t+r}yds\leq a_{3}\text{ if }t\geq t_{0},
\]
where $a_{i}>0$. Then%
\[
y\left(  t+r\right)  \leq\left(  \frac{a_{3}}{r}+a_{2}\right)  e^{a_{1}}.
\]

\end{lemma}

We apply Lemma \ref{UGL} with $y\left(  s\right)  =\left\Vert u^{n}\left(
s\right)  \right\Vert _{H_{0}^{1}\left(  \Omega\right)  }^{2}+2\left(
F\left(  u^{n}\left(  s\right)  \right)  ,1\right)  -2\left(  h,u^{n}\left(
s\right)  \right)  +M$ (where $M>0$ is such that $y\left(  s\right)  >0$),
$g\left(  s\right)  \equiv1$ and $w\left(  s\right)  \equiv2\widetilde{D}%
_{2}+\frac{2}{\lambda_{1}}\left\Vert h\right\Vert ^{2}.$ Then%
\begin{align}
&  \left\Vert u^{n}\left(  t+r\right)  \right\Vert _{H_{0}^{1}\left(
\Omega\right)  }^{2}+2\left(  F\left(  u^{n}\left(  t+r\right)  \right)
,1\right)  -2\left(  h,u^{n}\left(  t+r\right)  \right) \label{AcotH1F}\\
&  \leq K_{7}(\frac{e^{-\lambda_{1}t}\left\Vert u_{0}^{n}\right\Vert ^{2}%
+1}{r}+1+r)e^{r}\text{ for all }0\leq t\leq t+r.\nonumber
\end{align}
Using (\ref{PropF}) and
\[
2\left(  h,u^{n}\left(  t+r\right)  \right)  \leq\frac{2}{\lambda_{1}%
}\left\Vert h\right\Vert ^{2}+\frac{1}{2}\left\Vert u^{n}\left(  t+r\right)
\right\Vert _{H_{0}^{1}\left(  \Omega\right)  }^{2},
\]
we have%
\begin{align}
&  \left\Vert u^{n}\left(  t+r\right)  \right\Vert _{H_{0}^{1}\left(
\Omega\right)  }^{2}\label{AcotH1}\\
&  \leq K_{8}\left(  (\frac{e^{-\lambda_{1}t}\left\Vert u_{0}^{n}\right\Vert
^{2}+1}{r}+1+r)e^{r}\right)  \text{ for all }0\leq t\leq t+r.\nonumber
\end{align}
Therefore, the sequence $u^{n}\left(  \text{\textperiodcentered}\right)  $ is
bounded in $L^{\infty}\left(  r,T;H_{0}^{1}\left(  \Omega\right)  \right)  $
for all $0<r<T$.

Integrating (\ref{Inequt}) over $\left(  r,T\right)  $ and using
(\ref{AcotH1F}), (\ref{PropF}) we have%
\begin{align}
&  2\int_{r}^{T}\left\Vert u_{t}^{n}\right\Vert ^{2}dt+\left\Vert u^{n}\left(
T\right)  \right\Vert _{H_{0}^{1}}^{2}\label{AcotDerivL2}\\
&  \leq\left\Vert u^{n}\left(  r\right)  \right\Vert _{H_{0}^{1}\left(
\Omega\right)  }^{2}+2\left(  F\left(  u^{n}\left(  r\right)  \right)
,1\right)  -2\left(  h,u^{n}(r)\right) \nonumber\\
&  -2\left(  F\left(  u^{n}\left(  T\right)  \right)  ,1\right)  +2\left(
h,u^{n}\left(  T\right)  \right) \nonumber\\
&  \leq K_{7}(\frac{\left\Vert u_{0}^{n}\right\Vert ^{2}+1}{r}+1+r)e^{r}%
+\frac{2}{\lambda_{1}}\left\Vert h\right\Vert ^{2}+\frac{\lambda_{1}}%
{2}\left\Vert u^{n}\left(  T\right)  \right\Vert ^{2}+\widetilde{R}\nonumber\\
&  \leq K_{9}\left(  \frac{\left\Vert u_{0}^{n}\right\Vert ^{2}+1}%
{r}+1+r\right)  e^{r}+\frac{1}{2}\left\Vert u^{n}\left(  T\right)  \right\Vert
_{H_{0}^{1}}^{2},\nonumber
\end{align}
so that $u_{t}^{n}$ is bounded in $L^{2}\left(  r,T;L^{2}\left(
\Omega\right)  \right)  $ for all $0<r<T$.

On the other hand, from (\ref{Acotf}) and (\ref{AcotH1}) we get
\begin{align*}
\int_{r}^{T}\int_{\Omega}\left\vert f\left(  u^{n}\right)  \right\vert
^{2}dxdt  &  \leq K_{10}\int_{r}^{T}\left(  1+\left\Vert u^{n}\left(
t\right)  \right\Vert _{H_{0}^{1}\left(  \Omega\right)  }^{6}\right)  dt\\
&  \leq K_{11}\left(  T-r\right)  \left(  \frac{\left\Vert u_{0}%
^{n}\right\Vert ^{2}+1}{r}+1+r\right)  ^{3}e^{3r}.
\end{align*}
Then the equality $\Delta u^{n}=u_{t}^{n}+f\left(  u^{n}\right)  -h$ implies
that
\begin{equation}
\int_{r}^{T}\left\Vert \Delta u^{n}\right\Vert ^{2}dt\leq K_{12}\left(
T-r+1\right)  \left(  \frac{\left\Vert u_{0}^{n}\right\Vert ^{2}+1}%
{r}+1+r\right)  ^{3}e^{3r}, \label{AcotLapl}%
\end{equation}
ans then $u^{n}$ is bounded in $L^{2}\left(  r,T;D\left(  A\right)  \right)  $
for all $0<r<T$.

We note also that the compact embedding $H_{0}^{1}(\Omega)\subset L^{2}\left(
\Omega\right)  $ implies that for any $t>0$ the sequence $u^{n}\left(
t\right)  $ is precompact in $L^{2}\left(  \Omega\right)  $. Hence, applying
the Ascoli-Arzel\`{a} theorem we obtain, passing to a subsequence and using a
diagonal argument, that there exists a function $u:[0,+\infty)\rightarrow
L^{2}\left(  \Omega\right)  $ such that for all $0<r<T,$%
\begin{align}
u^{n}  &  \rightarrow u\text{ weakly star in }L^{\infty}\left(  r,T;H_{0}%
^{1}\left(  \Omega\right)  \right)  ,\label{ConvergRegular}\\
u^{n}  &  \rightarrow u\text{ in }C([r,T],L^{2}\left(  \Omega\right)
),\nonumber\\
u^{n}  &  \rightarrow u\text{ weakly in }L^{2}\left(  r,T;D\left(  A\right)
\right)  ,\nonumber\\
u_{t}^{n}  &  \rightarrow u_{t}\text{ weakly in }L^{2}\left(  r,T;L^{2}\left(
\Omega\right)  \right)  .\nonumber
\end{align}
Also, by a standard argument we obtain that for any sequence $t_{n}\rightarrow
t_{0}>0$ we have%
\begin{equation}
u^{n}\left(  t_{n}\right)  \rightarrow u\left(  t_{0}\right)  \text{ weakly in
}H_{0}^{1}\left(  \Omega\right)  . \label{ConvergRegular2}%
\end{equation}
On the other hand, by (\ref{IneqUn}) $u^{n}$ is bounded in $L^{\infty}\left(
0,T;L^{2}\left(  \Omega\right)  \right)  \cap L^{2}\left(  0,T;H_{0}%
^{1}\left(  \Omega\right)  \right)  \cap L^{4}\left(  0,T;L^{4}\left(
\Omega\right)  \right)  $ for all $T>0,$ and by (\ref{1}) and (\ref{2})
$u_{t}^{n}$ is bounded in $L^{2}\left(  0,T;H^{-1}\left(  \Omega\right)
\right)  +L^{\frac{4}{3}}\left(  0,T;L^{\frac{4}{3}}\left(  \Omega\right)
\right)  $ for all $T>0$. We note that $H_{0}^{1}\left(  \Omega\right)
\subset L^{4}\left(  \Omega\right)  \subset L^{\frac{4}{3}}\left(
\Omega\right)  \subset H^{-1}\left(  \Omega\right)  $ with continuous
embeddings, the first one being compact. Hence, by the Compactness Theorem
\cite{Lions},
\begin{align}
u^{n}  &  \rightarrow u\text{ weakly star in }L^{\infty}\left(  0,T;L^{2}%
\left(  \Omega\right)  \right)  ,\label{Converg}\\
u^{n}  &  \rightarrow u\text{ weakly in }L^{2}\left(  0,T;H_{0}^{1}\left(
\Omega\right)  \right)  ,\nonumber\\
u^{n}  &  \rightarrow u\text{ weakly in }L^{4}\left(  0,T;L^{4}\left(
\Omega\right)  \right)  ,\nonumber\\
u^{n}  &  \rightarrow u\text{ strongly in }L^{2}\left(  0,T;L^{2}%
(\Omega)\right)  ,\nonumber\\
u_{t}^{n}  &  \rightarrow u_{t}\text{ weakly in }L^{\frac{4}{3}}\left(
0,T;H^{-1}\left(  \Omega\right)  \right)  ,\text{ for all }T>0.\nonumber
\end{align}
Again by the Ascoli-Arzel\`{a} theorem we have that
\[
u^{n}\rightarrow u\text{ in }C([0,T],H^{-1}\left(  \Omega\right)  )\text{ for
all }T>0,
\]
and for any sequence $t_{n}\rightarrow t_{0}\geq0,$
\begin{equation}
u^{n}\left(  t_{n}\right)  \rightarrow u\left(  t_{0}\right)  \text{ weakly in
}L^{2}\left(  \Omega\right)  . \label{Converg2}%
\end{equation}
The last convergence implies that $u\left(  0\right)  =u_{0}.$

From the boundedness of $u^{n}$ in $L^{4}\left(  0,T;L^{4}\left(
\Omega\right)  \right)  $ and (\ref{2}) we have that $f\left(  u^{n}\right)  $
is bounded in $L^{\frac{4}{3}}\left(  0,T;L^{\frac{4}{3}}\left(
\Omega\right)  \right)  $. Since $u^{n}\left(  t,x\right)  \rightarrow
u\left(  t,x\right)  $, we have $f\left(  u^{n}\left(  t,x\right)  \right)
\rightarrow f(u\left(  t,x\right)  )$ for a.a. $\left(  t,x\right)  $, and
then $f\left(  u^{n}\right)  \rightarrow f\left(  u\right)  $ weakly in
$L^{\frac{4}{3}}\left(  0,T;L^{\frac{4}{3}}\left(  \Omega\right)  \right)  $
(see \cite[p.12]{Lions}).

In a standard way we can check then that $u\left(  \text{\textperiodcentered
}\right)  $ is a weak solution of (\ref{1}). Moreover, by the previous
arguments it is clear that $u$ is a regular solution.

Finally, we note that in view of (\ref{DerivL2}) and (\ref{D(A)}) all the
previous arguments leading to (\ref{IneqUn2}), (\ref{AcotH1}),
(\ref{AcotDerivL2}) and (\ref{AcotLapl}) are correct for any regular solution
with initial value in $L^{2}\left(  \Omega\right)  $. Thus, (\ref{PropReg1}%
)-(\ref{PropReg4}) follow. By \cite[p.102]{SellYou} we obtain that
(\ref{PropReg5})-(\ref{PropReg6}) hold. The energy equality (\ref{Energy})
follows from (\ref{Inequt}).
\end{proof}

\bigskip

Let
\[
K_{r}^{+}=\{u\left(  \text{\textperiodcentered}\right)  :u\text{ is a regular
solution of (\ref{1})}\}\text{.}%
\]
We define now the map $G_{r}:\mathbb{R}^{+}\times L^{2}\left(  \Omega\right)
\rightarrow P\left(  L^{2}\left(  \Omega\right)  \right)  $ by
\[
G_{r}(t,u_{0})=\{u\left(  t\right)  :u\in K_{r}^{+}\text{ and }u\left(
0\right)  =u_{0}\}.
\]
The set $K_{r}^{+}$ satisfies conditions $\left(  K1\right)  -\left(
K2\right)  $, so that by Lemma \ref{MS} $G_{r}$ is a multivalued semiflow.

\bigskip

\begin{remark}
In this case we are not able to prove that the semiflow $G_{r}$ is strict. The
reason is that if we take $u_{1},u_{2}\in K_{r}^{+}$ with $u_{2}\left(
0\right)  =u_{1}\left(  s\right)  $ and concatenate them, that is,%
\[
u\left(  t\right)  =\left\{
\begin{array}
[c]{c}%
u_{1}\left(  t\right)  \text{ if }0\leq t\leq s,\\
u_{2}\left(  t-s\right)  \text{ if }s\leq t,
\end{array}
\right.
\]
then $u$ is a weak solution of (\ref{1}), but we cannot state that it is
regular, as properties (\ref{LInfH1})-(\ref{DerivL2}) can fail now. Hence,
condition $\left(  K3\right)  $ is not known to be true.
\end{remark}

We shall obtain further some properties of the semiflow $G_{r}.$

\begin{lemma}
\label{ContLemma}Assume that (\ref{2}) holds. Let $\{u^{n}\}\subset K_{r}^{+}$
be a sequence such that $u^{n}(0)\rightarrow u_{0}$ weakly in $L^{2}(\Omega)$.
Then there exists a subsequence (denoted again by $u^{n}$), and a regular
solution of (\ref{1}) $u\in K_{r}^{+}$ satisfying $u(0)=u_{0},$ such that for
any sequence of times $t_{n}\geq0$ such that $t_{n}\rightarrow t_{0}$ we have
$u^{n}(t_{n})\rightarrow u(t_{0})$ weakly in $L^{2}(\Omega)$. \newline Also,
if $t_{0}>0$, then $u^{n}(t_{n})\rightarrow u(t_{0})$ strongly in $H_{0}%
^{1}\left(  \Omega\right)  $. \newline Moreover, if $u^{n}(0)\rightarrow
u_{0}$ strongly in $L^{2}(\Omega)$, then for $t_{n}\searrow0$ we get
$u^{n}(t_{n})\rightarrow u_{0}$ strongly in $L^{2}(\Omega)$.
\end{lemma}

\begin{proof}
Arguing as in the proof of Theorem \ref{ExistenceRegSol} we obtain the
existence of a subsequence of $u^{n}$ and a weak solution $u$ of (\ref{1})
with $u\left(  0\right)  =u_{0}$ such that the convergences
(\ref{ConvergRegular}), (\ref{ConvergRegular2}), (\ref{Converg}),
(\ref{Converg2}) hold. Hence, $u\in K_{r}^{+}.$

It follows that if $t_{0}>0$ and $t_{n}\rightarrow t_{0}$, then $u^{n}%
(t_{n})\rightarrow u(t_{0})$ strongly in $L^{2}(\Omega)$ and weakly in
$H_{0}^{1}\left(  \Omega\right)  $. We shall prove that $u^{n}(t_{n}%
)\rightarrow u(t_{0})$ strongly in $H_{0}^{1}\left(  \Omega\right)  $. Let
$t_{n},t_{0}\in(r,T)$. Multiplying (\ref{1}) by $u_{t}^{n}$ and using
(\ref{DerivH1}) we obtain
\begin{align*}
\frac{1}{2}\left\Vert u_{t}^{n}\right\Vert ^{2}+\frac{1}{2}\frac{d}%
{dt}\left\Vert u^{n}\right\Vert _{H_{0}^{1}\left(  \Omega\right)  }^{2}  &
\leq\left\Vert f\left(  u^{n}\right)  \right\Vert ^{2}+\left\Vert h\right\Vert
^{2},\\
\frac{1}{2}\left\Vert u_{t}\right\Vert ^{2}+\frac{1}{2}\frac{d}{dt}\left\Vert
u\right\Vert _{H_{0}^{1}\left(  \Omega\right)  }^{2}  &  \leq\left\Vert
f\left(  u\right)  \right\Vert ^{2}+\left\Vert h\right\Vert ^{2},
\end{align*}
and then by (\ref{2}) and (\ref{PropReg1}) we obtain
\begin{align*}
\left\Vert u^{n}\left(  t\right)  \right\Vert _{H_{0}^{1}\left(
\Omega\right)  }^{2}  &  \leq\left\Vert u^{n}\left(  s\right)  \right\Vert
_{H_{0}^{1}\left(  \Omega\right)  }^{2}+C\left(  t-s\right)  ,\\
\left\Vert u\left(  t\right)  \right\Vert _{H_{0}^{1}\left(  \Omega\right)
}^{2}  &  \leq\left\Vert u\left(  s\right)  \right\Vert _{H_{0}^{1}\left(
\Omega\right)  }^{2}+C\left(  t-s\right)  ,\text{ for all }r\leq s\leq t\leq
T,
\end{align*}
for some $C>0.$ Therefore the functions $J_{n}(t)=\left\Vert u^{n}\left(
s\right)  \right\Vert _{H_{0}^{1}\left(  \Omega\right)  }^{2}+Ct$,
$J(t)=\left\Vert u\left(  s\right)  \right\Vert _{H_{0}^{1}\left(
\Omega\right)  }^{2}+Ct$, are continuous and non-increasing in $[r,T]$.
Moreover, (\ref{ConvergRegular}) and the Compactness Theorem \cite{Lions}
imply that $J_{n}\left(  t\right)  \rightarrow J\left(  t\right)  $ for a.a.
$t\in\left(  r,T\right)  $. Take $r<t_{m}<t_{0}$ such that $t_{m}\rightarrow
t_{0}$ and $J_{n}\left(  t_{m}\right)  \rightarrow J_{n}\left(  t_{m}\right)
$ for all $m$. Then%
\[
J_{n}\left(  t_{n}\right)  -J\left(  t_{0}\right)  \leq J_{n}\left(
t_{m}\right)  -J\left(  t_{0}\right)  \leq\left\vert J_{n}\left(
t_{m}\right)  -J\left(  t_{m}\right)  \right\vert +\left\vert J\left(
t_{m}\right)  -J\left(  t_{0}\right)  \right\vert ,
\]
if $t_{n}\geq t_{m}$, so that for any $\varepsilon>0$ there exist $m\left(
\varepsilon\right)  $ and $N\left(  m\right)  $ such that $J_{n}\left(
t_{0}\right)  -J\left(  t_{0}\right)  \leq\varepsilon$ if $n\geq N$. Then
$\lim\ \sup J\left(  t_{n}\right)  \leq\lim\ \sup J\left(  t_{0}\right)  $, so
that $\lim\ \sup\left\Vert u^{n}\left(  t\right)  \right\Vert _{H_{0}%
^{1}\left(  \Omega\right)  }^{2}\leq\left\Vert u\left(  t\right)  \right\Vert
_{H_{0}^{1}\left(  \Omega\right)  }^{2}$. As $u^{n}(t_{n})\rightarrow
u(t_{0})$ weakly in $H_{0}^{1}\left(  \Omega\right)  $ implies $\lim
\ \inf\left\Vert u^{n}\left(  t_{n}\right)  \right\Vert _{H_{0}^{1}\left(
\Omega\right)  }^{2}\geq\left\Vert u\left(  t\right)  \right\Vert _{H_{0}%
^{1}\left(  \Omega\right)  }^{2}$, we obtain
\[
\left\Vert u^{n}\left(  t_{n}\right)  \right\Vert _{H_{0}^{1}\left(
\Omega\right)  }^{2}\rightarrow\left\Vert u\left(  t\right)  \right\Vert
_{H_{0}^{1}\left(  \Omega\right)  }^{2}\text{, }%
\]
so that $u^{n}(t_{n})\rightarrow u(t_{0})$ strongly in $H_{0}^{1}\left(
\Omega\right)  $.

Finally, let $u^{n}(0)\rightarrow u_{0}$ strongly in $L^{2}(\Omega)$. In view
of (\ref{IneqUn}) we have for some $C>0,$%
\begin{align*}
\frac{d}{dt}\left\Vert u^{n}\right\Vert ^{2}  &  \leq C,\\
\frac{d}{dt}\left\Vert u\right\Vert ^{2}  &  \leq C,\text{ for all }t\geq0,
\end{align*}
so that the functions $\overline{J}_{n}(t)=\left\Vert u^{n}\left(  s\right)
\right\Vert ^{2}+Ct$, $\overline{J}(t)=\left\Vert u\left(  s\right)
\right\Vert ^{2}+Ct$, are continuous and non-increasing for $t\geq0$. Hence,%
\[
\overline{J}_{n}(t_{n})-\overline{J}(0)\leq\overline{J}_{n}(0)-\overline
{J}(0)\rightarrow0,
\]
so that $\lim\ \sup\overline{J}\left(  t_{n}\right)  \leq\lim\ \sup
\overline{J}\left(  0\right)  $ and by the same argument as before we obtain
that $u^{n}(t_{n})\rightarrow u_{0}$ strongly in $L^{2}(\Omega)$.
\end{proof}

\bigskip

\begin{corollary}
\label{K4Regular}If (\ref{2}) holds, then $K_{r}^{+}$ satisfies condition
$\left(  K4\right)  .$
\end{corollary}

By standar arguments from Lemma \ref{ContLemma} the following result follows.

\begin{corollary}
\label{USC}Let (\ref{2}) hold. Then the multivalued semiflow $G_{r}$ has
compact values and the map $u_{0}\mapsto G_{r}\left(  t,u_{0}\right)  $ is
upper semicontinuous for all $t\geq0$, that is, for any neighborhood $O\left(
G_{r}\left(  t,u_{0}\right)  \right)  $ in $L^{2}\left(  \Omega\right)  $
there exists $\delta>0$ such that if $\left\Vert u_{0}-v_{0}\right\Vert
<\delta$, then $G_{r}\left(  t,v_{0}\right)  \subset O.$
\end{corollary}

\begin{lemma}
\label{AbsBall}Let (\ref{2}) hold. Then the ball
\[
B_{0}=\{u\in H_{0}^{1}\left(  \Omega\right)  :\left\Vert u\right\Vert
_{H_{0}^{1}\left(  \Omega\right)  }^{2}\leq4eR_{1}\},
\]
where $R_{1}>0$ is taken from (\ref{PropReg1}), is a compact absorbing\ for
$G_{r}$, that is, for any set $B$ bounded in $L^{2}\left(  \Omega\right)  $
there exists $T(B)$ such that%
\[
G_{r}\left(  t,B\right)  \subset B_{0}\text{ for }t\geq T\left(  B\right)  .
\]

\end{lemma}

\begin{proof}
The fact that $B_{0}$ is absorbing follows from (\ref{PropReg1}) by taking
$r=1$. Since $B_{0}$ is closed in $L^{2}\left(  \Omega\right)  $ and bounded
in $H_{0}^{1}\left(  \Omega\right)  $, the compacity in $L^{2}\left(
\Omega\right)  $ follows.
\end{proof}

\bigskip

Now we are ready to prove the existence of a global compact attractor.

\begin{theorem}
\label{AttrReg}Let (\ref{2}) hold. Then the multivalued semiflow $G_{r}$
posseses a global compact attractor $\Theta_{r}$. Moreover, for any set $B$
bounded in $L^{2}\left(  \Omega\right)  $ we have%
\begin{equation}
dist_{H_{0}^{1}\left(  \Omega\right)  }\left(  G_{r}(t,B\right)  ,\Theta
_{r}\mathcal{)}\rightarrow0\text{ as }t\rightarrow+\infty, \label{ConvergH1}%
\end{equation}
and also that $\Theta_{r}$ is compact in $H_{0}^{1}\left(  \Omega\right)  .$
\end{theorem}

\begin{proof}
The existence of a global compact attractor $\Theta_{r}$ follows from
Corollary \ref{USC}, Lemma \ref{AbsBall} and Theorem 4 in
\cite{MelnikValero98}.

Let now suppose that (\ref{ConvergH1}) is not true. Then there exists
$\varepsilon>0$ and a sequence $y_{n}\in G_{r}\left(  t_{n},B\right)  $ with
$t_{n}\rightarrow+\infty$ such that
\[
dist_{H_{0}^{1}\left(  \Omega\right)  }(y_{n},\Theta_{r}\mathcal{)>\varepsilon
}\text{ for all }n.
\]
There exists $y\in\Theta_{r}$ and a subsequence $y_{n_{k}}$ such that
$y_{n_{k}}\rightarrow y$ in $L^{2}\left(  \Omega\right)  $. In view of Lemma
\ref{ContLemma} the set $G_{r}(t,B)$ is precompact in $H_{0}^{1}\left(
\Omega\right)  $ for any $t>0$ and any bounded set in $L^{2}\left(
\Omega\right)  $. Therefore, from $y_{n_{k}}\in G_{r}\left(  1,G_{r}\left(
t_{n}-1,B\right)  \right)  \subset G_{r}(1,B_{0})$ for $n$ great enough, we
obtain that $y_{n_{k}}\rightarrow y$ in $H_{0}^{1}\left(  \Omega\right)  $,
which is a contradiction.

Finally, $\Theta_{r}\subset G_{r}(1,\Theta_{r})$ implies, by the same reason,
that $\Theta_{r}$ is precompact in $H_{0}^{1}\left(  \Omega\right)  $.
Moreover, since $\Theta_{r}$ is closed in $L^{2}\left(  \Omega\right)  $, it
is closed in $H_{0}^{1}\left(  \Omega\right)  $, as well, so that $\Theta_{r}$
is compact in $H_{0}^{1}\left(  \Omega\right)  .$
\end{proof}

\begin{remark}
Since $G_{r}\left(  t,u_{0}\right)  \subset G\left(  t,u_{0}\right)  $, it is
clear that $\Theta_{r}\subset\Theta.$
\end{remark}

\bigskip

The map $\gamma:\mathbb{R}\rightarrow L^{2}\left(  \Omega\right)  $ is called
a complete trajectory of $K_{r}^{+}$ if $\gamma\left(
\text{\textperiodcentered}+h\right)  \mid_{\lbrack0,+\infty)}\in K_{r}^{+}$
for any $h\in\mathbb{R}.$ The set of all complete trajectories of $K_{r}^{+}$
will be denoted by $\mathbb{F}_{r}$. Let $\mathbb{K}_{r}$ be the set of all
complete trajectories which are bounded in $L^{2}\left(  \Omega\right)  $, and
let $\mathbb{K}_{r}^{1}$ be the set of all complete trajectories which are
bounded in $H_{0}^{1}\left(  \Omega\right)  .$

\begin{lemma}
\label{EqBoundedComplete}Let (\ref{2}) hold. Then $\mathbb{K}_{r}%
=\mathbb{K}_{r}^{1}.$
\end{lemma}

\begin{proof}
It is clear that $\mathbb{K}_{r}^{1}\subset\mathbb{K}_{r}.$ Let $\gamma\left(
\text{\textperiodcentered}\right)  \in\mathbb{K}_{r}$ and denote $B_{\gamma
}=\cup_{t\in\mathbb{R}}\gamma\left(  t\right)  $. Then
\[
B_{\gamma}\subset G_{r}\left(  1,B_{\gamma}\right)
\]
and (\ref{PropReg1}) implies that $B_{\gamma}$ is bounded in $H_{0}^{1}\left(
\Omega\right)  ,$ so $\gamma\left(  \text{\textperiodcentered}\right)
\in\mathbb{K}_{r}^{1}.$
\end{proof}

\bigskip

Further, we shall prove that $\Theta_{r}$ is the union of all points lying in
a bounded complete trajectory, that is,
\begin{align}
\Theta_{r}  &  =\{\gamma\left(  0\right)  :\gamma\left(
\text{\textperiodcentered}\right)  \in\mathbb{K}_{r}\}=\{\gamma\left(
0\right)  :\gamma\left(  \text{\textperiodcentered}\right)  \in\mathbb{K}%
_{r}^{1}\}\label{BoundedComplete}\\
&  =\cup_{t\in\mathbb{R}}\{\gamma\left(  t\right)  :\gamma\left(
\text{\textperiodcentered}\right)  \in\mathbb{K}_{r}\}=\cup_{t\in\mathbb{R}%
}\{\gamma\left(  t\right)  :\gamma\left(  \text{\textperiodcentered}\right)
\in\mathbb{K}_{r}^{1}\}.\nonumber
\end{align}

\begin{theorem}
\label{CaractAttrR}Let (\ref{2}) hold. Then equalities (\ref{BoundedComplete})
hold true.
\end{theorem}

\begin{proof}
It follows from Theorem \ref{CaractAttrAbs}, Corollary \ref{K4Regular} and
Lemma \ref{EqBoundedComplete}.

\bigskip
\end{proof}

We shall establish now for $G_{r}$ the same statements as in Lemma
\ref{EqFixed} and in the first point of Theorem \ref{PropK+}.

\begin{lemma}
\label{EqFixedReg}Let (\ref{2}) hold. Then $\mathfrak{R}=\mathfrak{R}%
_{K_{r}^{+}}.$
\end{lemma}

\begin{proof}
Let $u_{0}\in\mathfrak{R}_{K_{r}^{+}}$. Then $u\left(  t\right)  \equiv u_{0}$
belongs to $K_{r}^{+}$. Thus $u\left(  \text{\textperiodcentered}\right)  $
satisfies (\ref{3}), so that (\ref{Stationary}) holds. Conversely, let
$v_{0}\in\mathfrak{R}$. In view of Lemma \ref{lem:1} the set of stationary
points of (\ref{1}) $\mathfrak{R}$ is bounded in $H^{2}\left(  \Omega\right)
\cap H_{0}^{1}\left(  \Omega\right)  $. Then the function $v\left(
\text{\textperiodcentered}\right)  $ defined by $v\left(  t\right)  \equiv
v_{0}$ belongs to $K_{r}^{+}$.
\end{proof}

\begin{corollary}
$\mathfrak{R\subset}\Theta_{r}\mathfrak{.}$
\end{corollary}

\begin{lemma}
\label{FixedPointsCarR}Let (\ref{2}) hold. Then $z\in\mathfrak{R}$ if and only
if $z\in G_{r}\left(  t,z\right)  $ for all $t\geq0.$
\end{lemma}

\begin{proof}
If $z\in\mathfrak{R}$, Lemma \ref{EqFixedReg} implies that the function
$v\left(  \text{\textperiodcentered}\right)  $ defined by $v\left(  t\right)
\equiv z$ belongs to $K_{r}^{+}$. Then $z\in G_{r}\left(  t,z\right)  ,$ for
any $t\geq0.$

Conversely, let $z\in G_{r}\left(  t,z\right)  $ for all $t\geq0$. Since $z\in
G_{r}\left(  t,z\right)  \subset G\left(  t,z\right)  $, Theorem \ref{PropK+}
implies that $z\in\mathfrak{R.}$
\end{proof}

\bigskip

It is clear that if the map $\gamma:\mathbb{R}\rightarrow L^{2}\left(
\Omega\right)  $ is a complete trajectory of $K_{r}^{+}$, then
\begin{equation}
\gamma\left(  t+s\right)  \in G_{r}\left(  t,\gamma\left(  s\right)  \right)
\text{ for all }s\in\mathbb{R}\text{ and }t\geq0. \label{CompleteTrajR}%
\end{equation}
Let $\gamma\left(  \text{\textperiodcentered}\right)  $ satisfy
(\ref{CompleteTrajR}). Then $\gamma\left(  t\right)  \in G\left(
t,\gamma\left(  s\right)  \right)  $ and by Theorem \ref{PropK+} we have that
$\gamma\left(  \text{\textperiodcentered}\right)  $ is a complete trajectory
for $K^{+}$. However, it is not clear whether $\gamma\left(
\text{\textperiodcentered}\right)  $ is a complete trajectory of $K_{r}^{+}$
or not, as $\left(  K3\right)  $ fails, so that we cannot use Lemma
\ref{CompleteTrajEquiv}. Nevertheless, we can obtain the following.

\begin{lemma}
\label{CompleteTrajEquivR}Let (\ref{2}) hold. Then the map $\gamma
:\mathbb{R}\rightarrow L^{2}\left(  \Omega\right)  $ is a complete trajectory
of $K_{r}^{+}$ if and only if $\gamma\in L_{loc}^{\infty}(\mathbb{R};H_{0}%
^{1}\left(  \Omega\right)  )$, $\gamma_{t}\in L_{loc}^{2}(\mathbb{R}%
;L^{2}\left(  \Omega\right)  )$ and (\ref{CompleteTrajR}) holds.
\end{lemma}

\begin{proof}
If the map $\gamma:\mathbb{R}\rightarrow L^{2}\left(  \Omega\right)  $ is a
complete trajectory of $K_{r}^{+}$, then clearly (\ref{CompleteTrajR}) and
$\gamma\in L_{loc}^{\infty}(\mathbb{R};H_{0}^{1}\left(  \Omega\right)  )$,
$\gamma_{t}\in L_{loc}^{2}(\mathbb{R};L^{2}\left(  \Omega\right)  )$ hold.

Conversely, let $\gamma\in L_{loc}^{\infty}(\mathbb{R};H_{0}^{1}\left(
\Omega\right)  )$, $\gamma_{t}\in L_{loc}^{2}(\mathbb{R};L^{2}\left(
\Omega\right)  )$ and (\ref{CompleteTrajR}) hold. Then by Theorem \ref{PropK+}
we have that $\gamma\left(  \text{\textperiodcentered}\right)  $ is a complete
trajectory for $K^{+}$. As $\gamma\in L_{loc}^{\infty}(\mathbb{R};H_{0}%
^{1}\left(  \Omega\right)  )$, $\gamma_{t}\in L_{loc}^{2}(\mathbb{R}%
;L^{2}\left(  \Omega\right)  )$, it is clear that $\gamma$ is a complete
trajectory of $K_{r}^{+}.$
\end{proof}

\bigskip

We shall prove now that
\begin{equation}
\Theta_{r}=M_{r}^{+}(\mathfrak{R})=M_{r}^{-}(\mathfrak{R}), \label{Str}%
\end{equation}
where
\[%
\begin{array}
[c]{c}%
M_{r}^{-}(\mathfrak{R})=\left\{  z\,:\,\exists\gamma(\cdot)\in\mathbb{K}%
_{r},\,\ \gamma(0)=z,\,\,\,\ \mathrm{dist}_{L^{2}(\Omega)}(\gamma
(t),\mathfrak{R})\rightarrow0,\,\ t\rightarrow+\infty\right\}  ,\\
M_{r}^{+}(\mathfrak{R})=\left\{  z\,:\,\exists\gamma(\cdot)\in\mathbb{F}%
_{r},\,\ \gamma(0)=z,\,\,\,\ \mathrm{dist}_{L^{2}(\Omega)}(\gamma
(t),\mathfrak{R})\rightarrow0,\,\ t\rightarrow-\infty\right\}  .
\end{array}
\]
As in the case of weak solutions, in the definition of $M_{r}^{+}%
(\mathfrak{R})$ we can replace $\mathbb{F}_{r}$ by $\mathbb{K}_{r}$, since
every $\gamma$ as given in the definition of $M_{r}^{+}(\mathfrak{R})$ belongs
to $\mathbb{K}_{r}$ in view of (\ref{PropReg2}).

\begin{theorem}
\label{StructureReg} Under conditions (\ref{2}) it holds
\begin{equation}
\Theta_{r}=M_{r}^{+}(\mathfrak{R})=M_{r}^{-}(\mathfrak{R}). \label{EqualAttR}%
\end{equation}
Moreover,
\begin{equation}%
\begin{array}
[c]{c}%
M_{r}^{-}(\mathfrak{R})=\left\{  z\,:\,\exists\gamma(\cdot)\in\mathbb{K}%
_{r},\,\ \gamma(0)=z,\,\,\,\ \mathrm{dist}_{H_{0}^{1}(\Omega)}(\gamma
(t),\mathfrak{R})\rightarrow0,\,\ t\rightarrow+\infty\right\}  ,\\
M_{r}^{+}(\mathfrak{R})=\left\{  z\,:\,\exists\gamma(\cdot)\in\mathbb{F}%
_{r},\,\ \gamma(0)=z,\,\,\,\ \mathrm{dist}_{H_{0}^{1}(\Omega)}(\gamma
(t),\mathfrak{R})\rightarrow0,\,\ t\rightarrow-\infty\right\}  .
\end{array}
\label{StrH1}%
\end{equation}

\end{theorem}

\begin{proof}
We can prove this theorem arguing in a rather similar way as in Theorem
\ref{teor2}. However, we shall prove it using the method of the Lyapunov
function, as in \cite{BabinVishik89}, \cite{Temam2}.

Let $z\in\Theta_{r}$ and let $u\in\mathbb{K}_{r}$ be such that $u(0)=z$. We
note that the energy function $E\left(  u\left(  t\right)  \right)  $, $t>0$,
given in (\ref{Energy}) is nonincreasing and bounded below (which follows
easily from (\ref{PropF})) for any $u\in K_{r}^{+}$. Hence, $E\left(  u\left(
t\right)  \right)  \rightarrow l$, as $t\rightarrow+\infty$, for some
$l\in\mathbb{R}.$

Suppose that there exist $\varepsilon>0$ and a sequence $u\left(
t_{n}\right)  $, $t_{n}\rightarrow+\infty$, such that
\[
\mathrm{dist}_{L^{2}(\Omega)}(u(t_{n}),\mathfrak{R})>\varepsilon.
\]
In view of Theorem \ref{AttrReg} we have that $\Theta_{r}$ is compact in
$H_{0}^{1}\left(  \Omega\right)  $, and then we can take a converging
subsequence (denoted again $u\left(  t_{n}\right)  $) for which $u\left(
t_{n}\right)  \rightarrow y$ in $H_{0}^{1}\left(  \Omega\right)  $, where
$t_{n}\rightarrow+\infty$. Since the function $E:H_{0}^{1}\left(
\Omega\right)  \rightarrow\mathbb{R}$ is continuous, we obtain that $E\left(
y\right)  =l$. We shall prove that $y\in\mathfrak{R}$. Fix $t>0$. In view of
Lemma \ref{ContLemma} there exists $v\in K_{r}^{+}$ and a subsequence of
$v_{n}\left(  \text{\textperiodcentered}\right)  =u\left(
\text{\textperiodcentered}+t_{n}\right)  $ (denoted $v_{n}$ again) such that
$v\left(  0\right)  =y$ and $v_{n}\left(  t\right)  \rightarrow v\left(
t\right)  =z$ in $H_{0}^{1}\left(  \Omega\right)  $. Hence, $E\left(
v_{n}\left(  t\right)  \right)  \rightarrow E(z)$ implies that $E\left(
z\right)  =l$. We note that $v\left(  \text{\textperiodcentered}\right)  $
satisfies (\ref{Energy}) for all $0\leq s\leq t$, so that%
\[
l+2\int_{0}^{t}\left\Vert v_{r}\right\Vert ^{2}dr=E\left(  z\right)
+2\int_{0}^{t}\left\Vert v_{r}\right\Vert ^{2}dr=E\left(  v\left(  0\right)
\right)  =E(y)=l\text{.}%
\]
This implies that $v_{r}\left(  r\right)  =0$ for a.a. $r$, and therefore
$y\in\mathfrak{R}$. Hence, we obtain a contradiction. Thus, $\Theta_{r}\subset
M_{r}^{-}(\mathfrak{R})$. The converse inclusion is obvious from Theorem
\ref{CaractAttrR}, so that $\Theta_{r}=M_{r}^{-}(\mathfrak{R}).$

On the other hand, we observe that for any $u\in\mathbb{F}_{r}$ equality
(\ref{Energy}) is satisfied for all $-\infty<s\leq t$. Let $z\in\Theta_{r}$
and let $u\in\mathbb{K}_{r}=\mathbb{K}_{r}^{1}$ be such that $u(0)=z$. In view
of (\ref{PropF}) the function $(F\left(  u\left(  t\right)  \right)  ,1)$ is
bounded above. Hence, $E\left(  u\left(  t\right)  \right)  \rightarrow l$, as
$t\rightarrow-\infty$, for some $l\in\mathbb{R}.$ As before, suppose that
there exist $\varepsilon>0$ and a sequence $u\left(  t_{n}\right)  $,
$t_{n}\rightarrow\infty$, such that
\[
\mathrm{dist}_{L^{2}(\Omega)}(u(-t_{n}),\mathfrak{R})>\varepsilon,
\]
and we have that $u\left(  -t_{n}\right)  \rightarrow y$ in $H_{0}^{1}\left(
\Omega\right)  $, $E\left(  y\right)  =l$. Also, for a fixed $t>0$ there
exists $v\in K_{r}^{+}$ and a subsequence of $v_{n}\left(
\text{\textperiodcentered}\right)  =u\left(  \text{\textperiodcentered}%
-t_{n}\right)  $ (denoted $v_{n}$ again) such that $v\left(  0\right)  =y$ and
$v_{n}\left(  t\right)  \rightarrow v\left(  t\right)  =z$ in $H_{0}%
^{1}\left(  \Omega\right)  $. Hence, $E\left(  v_{n}\left(  t\right)  \right)
\rightarrow E(z)$ implies that $E\left(  z\right)  =l$ and therefore, arguing
as before, $y\in\mathfrak{R}$, which is a contradiction. As before, we obtain
that $\Theta_{r}=M_{r}^{+}(\mathfrak{R}).$

Finally, let us prove that the convergence takes place in $H_{0}^{1}\left(
\Omega\right)  $. Let us suppose that there exist $\varepsilon>0$ and a
sequence $u\left(  t_{n}\right)  $, $t_{n}\rightarrow+\infty$, such that
\[
\mathrm{dist}_{H_{0}^{1}(\Omega)}(u(t_{n}),\mathfrak{R})>\varepsilon.
\]
In view of $\mathrm{dist}_{L^{2}(\Omega)}(u(t_{n}),\mathfrak{R})\rightarrow0$
and the compacity of $\mathfrak{R}$, we can extract a subsequence $u(t_{n_{k}%
})$ such that $u(t_{n_{k}})\rightarrow\overline{u}\in\mathfrak{R}$ in
$L^{2}\left(  \Omega\right)  $. It follows from the compacity of $\Theta_{r}$
in $H_{0}^{1}\left(  \Omega\right)  $ that in fact $u(t_{n_{k}})\rightarrow
\overline{u}\in\mathfrak{R}$ in $H_{0}^{1}\left(  \Omega\right)  $, which is a
contradiction. Hence, the first part of (\ref{StrH1}) holds. The second one is
proved in the same way.
\end{proof}

\bigskip

\section{An attractor in $H_{0}^{1}\left(  \Omega\right)  $. Existence and
structure of the global attractor for strong solutions\label{StrStrong}}

In this section we shall define a semiflow in the phase space $H_{0}%
^{1}\left(  \Omega\right)  $. For this aim we introduce now a stronger concept
of solution for (\ref{1}).

The function $u\in L_{loc}^{2}(0,+\infty;H_{0}^{1}(\Omega))\bigcap L_{loc}%
^{4}(0,+\infty;L^{4}(\Omega))$ is called a strong solution of (\ref{1}) on
$(0,+\infty)$ if for all $T>0,\,v\in H_{0}^{1}(\Omega)\,\ $and $\eta\in
C_{0}^{\infty}(0,T)$ we have that (\ref{Eq1}) holds and
\begin{align}
u  &  \in L^{\infty}\left(  0,T;H_{0}^{1}\left(  \Omega\right)  \right)
,\label{LInfStrong}\\
u_{t}  &  \in L^{2}\left(  0,T;L^{2}\left(  \Omega\right)  \right)
,\ \forall\text{ }T>0. \label{DerivStrong}%
\end{align}

Arguing as in Section \ref{StrReg} we obtain that any strong solution $u$
satisfies
\begin{equation}
u\in L^{2}\left(  0,T;D\left(  A\right)  \right)  . \label{L2DAStrong}%
\end{equation}

By Lemma \ref{lem:4} for any $u_{0}\in H_{0}^{1}\left(  \Omega\right)  $ there
exists at least one strong solution $u\left(  \text{\textperiodcentered
}\right)  $ such that $u\left(  0\right)  =u_{0}$. Moreover, any strong
solution satisfies good properties, as given in the following lemma.

\begin{lemma}
\label{PropStrongSol}Let (\ref{2}) hold. Then every strong solution of
(\ref{1}) satisfies the following properties:%
\begin{equation}
u\in C([0,+\infty),H_{0}^{1}\left(  \Omega\right)  ), \label{Cont}%
\end{equation}%
\begin{equation}
\frac{d}{dt}\left\Vert u\right\Vert _{H_{0}^{1}\left(  \Omega\right)  }%
^{2}=2\left(  -\Delta u,u_{t}\right)  \text{ for a.a. }t>0,
\label{DerivStrong2}%
\end{equation}%
\begin{equation}
E\left(  u\left(  t\right)  \right)  +2\int_{s}^{t}\left\Vert u_{r}\right\Vert
^{2}dr=E\left(  u\left(  s\right)  \right)  \text{, for all }t\geq s\geq0,
\label{Energy2}%
\end{equation}
where $E\left(  u\left(  t\right)  \right)  =\left\Vert u\left(  t\right)
\right\Vert _{H_{0}^{1}\left(  \Omega\right)  }^{2}+2\left(  F\left(  u\left(
t\right)  \right)  ,1\right)  -2\left(  h,u\left(  t\right)  \right)  .$ Also,
there exist $R_{1},R_{2}>0$ such that
\begin{equation}
\int_{0}^{t}\left\Vert u_{r}\right\Vert ^{2}dr+\left\Vert u\left(  t\right)
\right\Vert _{H_{0}^{1}\left(  \Omega\right)  }^{2}\leq R_{1}\left(
\left\Vert u_{0}\right\Vert _{H_{0}^{1}\left(  \Omega\right)  }^{4}+1\right)
\text{,} \label{PropStrong1}%
\end{equation}%
\begin{equation}
\int_{0}^{t}\left\Vert \Delta u^{n}\right\Vert ^{2}dt\leq R_{2}\left(
t+1\right)  \left(  1+\left\Vert u_{0}\right\Vert _{H_{0}^{1}\left(
\Omega\right)  }^{12}\right)  ,\text{ for all }t\geq0. \label{PropStrong2}%
\end{equation}

\end{lemma}

\begin{proof}
In view of (\ref{LInfStrong})-(\ref{DerivStrong}) by standard results
\cite[p.102]{SellYou}, we obtain that $u$ belongs to $C([0,+\infty),H_{0}%
^{1}\left(  \Omega\right)  )$ and (\ref{DerivH1}), (\ref{DerivF}) hold. We
multiply (\ref{1}) by $u_{t}$ and using (\ref{DerivH1}), (\ref{DerivF}) we
obtain%
\[
2\left\Vert u_{t}\right\Vert ^{2}+\frac{d}{dt}\left(  \left\Vert u\right\Vert
_{H_{0}^{1}\left(  \Omega\right)  }^{2}+2\left(  F\left(  u\right)  ,1\right)
-2\left(  h,u\right)  \right)  =0\text{ for a.a. }t\in\left(  0,T\right)  .
\]
Integrating over $\left(  s,t\right)  $ we have
\begin{align*}
&  2\int_{s}^{t}\left\Vert u_{r}\right\Vert ^{2}dr+\left\Vert u\left(
t\right)  \right\Vert _{H_{0}^{1}\left(  \Omega\right)  }^{2}+2\left(
F\left(  u\left(  t\right)  \right)  ,1\right)  -2\left(  h,u\left(  t\right)
\right)  \\
&  =\left\Vert u\left(  s\right)  \right\Vert _{H_{0}^{1}\left(
\Omega\right)  }^{2}+2\left(  F\left(  u\left(  s\right)  \right)  ,1\right)
-2\left(  h,u\left(  s\right)  \right)  .
\end{align*}
Thus we obtain (\ref{Energy2}). Due to (\ref{PropF}) we get
\begin{align*}
&  2\int_{0}^{t}\left\Vert u_{r}\right\Vert ^{2}dr+\left\Vert u\left(
t\right)  \right\Vert _{H_{0}^{1}\left(  \Omega\right)  }^{2}\\
&  \leq\left(  1+\frac{1}{\lambda_{1}}\right)  \left\Vert u_{0}\right\Vert
_{H_{0}^{1}\left(  \Omega\right)  }^{2}+\frac{\lambda_{1}}{2}\left\Vert
u\left(  t\right)  \right\Vert ^{2}+\widetilde{D}_{2}+2D_{1}\int_{\Omega
}\left(  1+\left\vert u_{0}\left(  x\right)  \right\vert ^{4}\right)
dx+\left(  \frac{2}{\lambda_{1}}+1\right)  \left\Vert h\right\Vert ^{2},
\end{align*}
so%
\[
\int_{0}^{t}\left\Vert u_{r}\right\Vert ^{2}dr+\left\Vert u\left(  t\right)
\right\Vert _{H_{0}^{1}\left(  \Omega\right)  }^{2}\leq R_{1}\left(
\left\Vert u_{0}\right\Vert _{H_{0}^{1}\left(  \Omega\right)  }^{4}+1\right)
\text{, for all }t\geq0.
\]
Finally, by%
\begin{align*}
\int_{0}^{T}\int_{\Omega}\left\vert f\left(  u\right)  \right\vert ^{2}dxdt &
\leq\int_{0}^{T}\left(  1+\left\Vert u\left(  t\right)  \right\Vert
_{H_{0}^{1}\left(  \Omega\right)  }^{6}\right)  dt\\
&  \leq T\left(  1+\left(  R_{1}\left(  \left\Vert u_{0}\right\Vert
_{H_{0}^{1}\left(  \Omega\right)  }^{4}+1\right)  \right)  ^{3}\right)
\end{align*}
and the equality $\Delta u^{n}=u_{t}^{n}+f\left(  u^{n}\right)  -h$ we obtain
that
\[
\int_{0}^{T}\left\Vert \Delta u^{n}\right\Vert ^{2}dt\leq R_{2}\left(
T+1\right)  \left(  1+\left\Vert u_{0}\right\Vert _{H_{0}^{1}\left(
\Omega\right)  }^{12}\right)  .
\]
\bigskip
\end{proof}

\bigskip

Let
\[
K_{s}^{+}=\{u\left(  \text{\textperiodcentered}\right)  :u\text{ is a strong
solution of (\ref{1})\}.}%
\]
We define now the map $G_{s}:\mathbb{R}^{+}\times H_{0}^{1}\left(
\Omega\right)  \rightarrow P\left(  H_{0}^{1}\left(  \Omega\right)  \right)  $
by
\[
G_{s}(t,u_{0})=\{u\left(  t\right)  :u\in K_{s}^{+}\text{ and }u\left(
0\right)  =u_{0}\}.
\]

We can check easily that $K_{s}^{+}$ satisfies conditions $\left(  K1\right)
-\left(  K3\right)  $. Then Lemma \ref{MS} implies the following.

\begin{lemma}
Let (\ref{2}) hold. Then $G_{s}$ is a strict multivalued semiflow.
\end{lemma}

\bigskip

We shall obtain further some properties of the semiflow $G_{s}.$

\begin{lemma}
\label{ContLemmaStrong}Assume that (\ref{2}) holds. Let $\{u^{n}\}\subset
K_{s}^{+}$ be a sequence such that $u^{n}(0)\rightarrow u_{0}$ weakly in
$H_{0}^{1}(\Omega)$. Then there exists a subsequence (denoted again by $u^{n}%
$), and a strong solution of (\ref{1}) $u\in K_{s}^{+}$ satisfying
$u(0)=u_{0},$ such that for any sequence of times $t_{n}\geq0$ such that
$t_{n}\rightarrow t_{0}$ we have $u^{n}(t_{n})\rightarrow u(t_{0})$ weakly in
$H_{0}^{1}(\Omega)$. \newline Also, if $t_{0}>0$, then $u^{n}(t_{n}%
)\rightarrow u(t_{0})$ strongly in $H_{0}^{1}\left(  \Omega\right)  $.
\newline Moreover, if $u^{n}(0)\rightarrow u_{0}$ strongly in $H_{0}%
^{1}(\Omega)$, then for $t_{n}\searrow0$ we get $u^{n}(t_{n})\rightarrow
u_{0}$ strongly in $H_{0}^{1}(\Omega)$.
\end{lemma}

\begin{proof}
Since obviously $K_{s}^{+}\subset K_{r}^{+}$, it follows from Lemma
\ref{ContLemma} the existence of a regular solution $u\left(
\text{\textperiodcentered}\right)  \in K_{r}^{+}$ with $u\left(  0\right)
=u_{0}$ and a subsequence such that (\ref{ConvergRegular}), (\ref{Converg})
hold and
\begin{align*}
u^{n}\left(  t_{n}\right)   &  \rightarrow u\left(  t_{0}\right)  \text{
strongly in }H_{0}^{1}\left(  \Omega\right)  \text{ if }t_{0}>0,\\
u^{n}\left(  t_{n}\right)   &  \rightarrow u\left(  t_{0}\right)  \text{
strongly in }L^{2}\left(  \Omega\right)  .
\end{align*}
Thus, (\ref{PropStrong1}) implies by a standard argument that
\[
u^{n}\left(  t_{n}\right)  \rightarrow u\left(  t_{0}\right)  \text{ weakly in
}H_{0}^{1}\left(  \Omega\right)  .
\]
It remains to check that $u$ is a strong solution. In view of
(\ref{PropStrong1})-(\ref{PropStrong2}) for all $T>0$ the sequence $u^{n}$ is
bounded in $L^{\infty}\left(  0,T;H_{0}^{1}\left(  \Omega\right)  \right)
\cap L^{2}(0,T;D\left(  A\right)  )$ and $u_{t}^{n}$ is bounded in
$L^{2}(0,T;L^{2}\left(  \Omega\right)  )$. Hence, $u^{n}\rightarrow u$ weakly
star in $L^{\infty}\left(  0,T;H_{0}^{1}\left(  \Omega\right)  \right)  $,
weakly in $L^{2}(0,T;D\left(  A\right)  )$ and $u_{t}^{n}\rightarrow u_{t}$
weakly in $L^{2}(0,T;L^{2}\left(  \Omega\right)  ),$ so that $u$ satisfies
(\ref{LInfStrong})-(\ref{DerivStrong}) and then $u\in K_{s}^{+}.$

Finally, if $u^{n}(0)\rightarrow u_{0}$ strongly in $H_{0}^{1}(\Omega)$, then
arguing in the same way as in Lemma \ref{ContLemma} we obtain that
$u^{n}(t_{n})\rightarrow u_{0}$ strongly in $H_{0}^{1}(\Omega)$ for
$t_{n}\searrow0.$
\end{proof}

\bigskip

\begin{corollary}
\label{K4Strong}Let (\ref{2}) hold. Then $K_{s}^{+}$ satisfies condition
$\left(  K4\right)  .$
\end{corollary}

\begin{corollary}
\label{USCStrong}Let (\ref{2}) hold. Then the multivalued semiflow $G_{s}$ has
compact values and the map $u_{0}\mapsto G_{s}\left(  t,u_{0}\right)  $ is
upper semicontinuous for all $t\geq0$.
\end{corollary}

We prove further the existence of a global compact attractor.

\begin{theorem}
\label{AttrStrong}Let (\ref{2}) hold. Then the multivalued semiflow $G_{s}$
posseses a global compact invariant attractor $\Theta_{s}$.
\end{theorem}

\begin{proof}
Since $G_{s}\left(  t,u_{0}\right)  \subset G_{r}\left(  t,u_{0}\right)  $,
the ball $B_{0}$ defined in Lemma \ref{AbsBall} is absorbing for $G_{s}$.
Also, the operator $G_{s}\left(  t,\text{\textperiodcentered}\right)  $ is
compact for $t>0$ by Lemma \ref{ContLemmaStrong}, so that $K=\overline
{G_{s}(1,B_{0})}^{H_{0}^{1}}$ is a compact absorbing set. Then using Corollary
\ref{USCStrong} the existence of a global compact minimal attractor
$\Theta_{s}$ follows from Theorem 4 in \cite{MelnikValero98}. As $G_{s}$ is
strict, it follows from Remark 8 in \cite{MelnikValero98} that $G_{s}\left(
t,\Theta_{s}\right)  =\Theta_{s}$ for all $t\geq0.$
\end{proof}

\bigskip

We will prove further that in fact the attractors $\Theta_{s}$ and $\Theta
_{r}$ coincide.

\begin{lemma}
\label{EqAttr}Let (\ref{2}) hold. Then $\Theta_{s}=\Theta_{r}.$
\end{lemma}

\begin{proof}
Since $G_{s}\left(  t,u_{0}\right)  \subset G_{r}\left(  t,u_{0}\right)  $, we
have that $\Theta_{s}\subset\Theta_{r}$.

Conversely, if $z\in\Theta_{r}$, then by (\ref{BoundedComplete}) we have that
$z=\gamma\left(  0\right)  $, where $\gamma\left(  \text{\textperiodcentered
}\right)  \in\mathbb{K}_{r}^{1}$, the set of all bounded (in $H_{0}^{1}\left(
\Omega\right)  $) complete trajectories corresponding to $K_{r}^{+}$. It is
easy to see that $\gamma\mid_{\lbrack\tau,+\infty)}$ is a strong solution for
any $\tau\in\mathbb{R}$. Hence, $z=\gamma\left(  0\right)  \in G_{s}\left(
t_{n},\gamma\left(  -t_{n}\right)  \right)  $ for $t_{n}\rightarrow+\infty$.
Hence,%
\[
dist_{H_{0}^{1}\left(  \Omega\right)  }\left(  z,\Theta_{s}\right)  \leq
dist_{H_{0}^{1}\left(  \Omega\right)  }\left(  G_{s}\left(  t_{n}%
,\gamma\left(  -t_{n}\right)  \right)  ,\Theta_{s}\right)  \rightarrow0,
\]
so that $z\in\Theta_{s}$.
\end{proof}

\bigskip

The map $\gamma:\mathbb{R}\rightarrow H_{0}^{1}\left(  \Omega\right)  $ is
called a complete trajectory of $G_{s}$ if $\gamma\left(
\text{\textperiodcentered}+h\right)  \mid_{\lbrack0,+\infty)}\in K_{s}^{+}$
for any $h\in\mathbb{R}.$ The set of all complete trajectories of $K_{s}^{+}$
will be denoted by $\mathbb{F}_{s}$. Let $\mathbb{K}_{s}$ be the set of all
complete trajectories which are bounded in $H_{0}^{1}\left(  \Omega\right)  $.

\begin{lemma}
\label{EqK}Let (\ref{2}) hold. Then $\mathbb{K}_{s}=\mathbb{K}_{r}%
^{1}=\mathbb{K}_{r}.$
\end{lemma}

\begin{proof}
$\mathbb{K}_{s}\subset\mathbb{K}_{r}^{1}$ is obvious, and $\mathbb{K}_{r}%
^{1}\subset\mathbb{K}_{s}$ follows from the arguments in the proof of Lemma
\ref{EqAttr}. The last equality was done in Lemma \ref{EqBoundedComplete}.
\end{proof}

\bigskip

We shall establish now the same statements of Lemma \ref{EqFixed} and Theorem
\ref{PropK+} for $G_{s}.$

First we can characterize the attractor $\Theta_{s}$ as the union of all
points lying in a bounded complete trajectory.

\begin{lemma}
\label{CaractAttrStrong}Let (\ref{2}) hold. Then we have%
\begin{equation}
\Theta_{s}=\{\gamma\left(  0\right)  :\gamma\left(  \text{\textperiodcentered
}\right)  \in\mathbb{K}_{s}\}=\cup_{t\in\mathbb{R}}\{\gamma\left(  t\right)
:\gamma\left(  \text{\textperiodcentered}\right)  \in\mathbb{K}_{s}\}.
\label{BoundedCompleteStrong}%
\end{equation}

\end{lemma}

\begin{proof}
As $K_{s}^{+}$ satisfies $\left(  K1\right)  -\left(  K4\right)  $, it is a
direct consequence of either Theorem \ref{CaractAttrAbs} or
\ref{CaractAttrAbs2}. Also, it follows from Lemmas \ref{EqAttr}, \ref{EqK} and
Theorem \ref{CaractAttrR}.
\end{proof}

\bigskip

\begin{lemma}
\label{EqFixedStrong}Let (\ref{2}) hold. Then $\mathfrak{R}=\mathfrak{R}%
_{K_{s}^{+}}.$
\end{lemma}

\begin{proof}
Let $u_{0}\in\mathfrak{R}_{K_{s}^{+}}$. Then $u\left(  t\right)  \equiv u_{0}$
belongs to $K_{s}^{+}$. Thus $u\left(  \text{\textperiodcentered}\right)  $
satisfies (\ref{3}), so that (\ref{Stationary}) holds. Conversely, let
$v_{0}\in\mathfrak{R}$. In view of Lemma \ref{lem:1} the set of stationary
points of (\ref{1}) $\mathfrak{R}$ is bounded in $H^{2}\left(  \Omega\right)
\cap H_{0}^{1}\left(  \Omega\right)  $. Then the function $v\left(
\text{\textperiodcentered}\right)  $ defined by $v\left(  t\right)  \equiv
v_{0}$ belongs to $K_{s}^{+}$.
\end{proof}

\begin{corollary}
Let (\ref{2}) hold. Then $\mathfrak{R\subset}\Theta_{s}\mathfrak{.}$
\end{corollary}

\begin{lemma}
\label{FixedPointsCarS}Let (\ref{2}) hold. Then $z\in\mathfrak{R}$ if and only
if $z\in G_{s}\left(  t,z\right)  $ for all $t\geq0.$
\end{lemma}

\begin{proof}
As $K_{s}^{+}$ satisfies $\left(  K1\right)  -\left(  K4\right)  $, it follows
from Lemma \ref{EqFixedStrong} and Lemma \ref{FixedPointsCar}.
\end{proof}

\begin{remark}
We can prove that $z\in G_{s}\left(  t,z\right)  ,$ for all $t\geq0,$ implies
$z\in\mathfrak{R}$ also by using the Lyapunov function $E\left(
\text{\textperiodcentered}\right)  $. Indeed, if $z\in G_{s}\left(
t,z\right)  ,$ for any $t\geq0,$ then for every $T>0$ there exists
$v^{T}\left(  \text{\textperiodcentered}\right)  \in K_{s}^{+}$ such that
$v^{T}\left(  T\right)  =z.$ Thus by the energy equality (\ref{Energy2}) we
have
\[
E(z)+2\int_{0}^{T}\left\Vert v_{r}^{T}\right\Vert ^{2}dr=E\left(  v^{T}\left(
T\right)  \right)  +2\int_{0}^{T}\left\Vert v_{r}^{T}\right\Vert
^{2}dr=E\left(  v\left(  0\right)  \right)  =E(z)\text{, }%
\]
so that $\int_{0}^{T}\left\Vert v_{r}^{T}\right\Vert ^{2}dr=0$. Therefore,
$v_{t}^{T}=0$ for a.a. $t$ in $\left(  0,T\right)  $ and $v^{T}\left(
t\right)  =z$ for all $t\in\lbrack0,T]$. Hence, $z\in\mathfrak{R}_{K_{s}^{+}%
}=\mathfrak{R.}$
\end{remark}

\begin{lemma}
\label{CompleteTrajEquivS}Let (\ref{2}) hold. Then the map $\gamma
:\mathbb{R}\rightarrow H_{0}^{1}\left(  \Omega\right)  $ is a complete
trajectory of $K_{s}^{+}$ if and only if
\begin{equation}
\gamma\left(  t+s\right)  \in G_{s}\left(  t,\gamma\left(  s\right)  \right)
\text{ for all }s\in\mathbb{R}\text{ and }t\geq0. \label{PropCompleteTrajS}%
\end{equation}

\end{lemma}

\begin{proof}
As $\left(  K1\right)  -\left(  K4\right)  $ hold, the result follows from
Lemma \ref{CompleteTrajEquiv}.
\end{proof}

\bigskip

We shall prove now that
\begin{equation}
\Theta_{s}=M_{s}^{+}(\mathfrak{R})=M_{s}^{-}(\mathfrak{R}), \label{Str2}%
\end{equation}
where
\[%
\begin{array}
[c]{c}%
M_{s}^{-}(\mathfrak{R})=\left\{  z\,:\,\exists\gamma(\cdot)\in\mathbb{K}%
_{s},\,\ \gamma(0)=z,\,\,\,\ \mathrm{dist}_{H_{0}^{1}(\Omega)}(\gamma
(t),\mathfrak{R})\rightarrow0,\,\ t\rightarrow+\infty\right\}  ,\\
M_{s}^{+}(\mathfrak{R})=\left\{  z\,:\,\exists\gamma(\cdot)\in\mathbb{F}%
_{s},\,\ \gamma(0)=z,\,\,\,\ \mathrm{dist}_{H_{0}^{1}(\Omega)}(\gamma
(t),\mathfrak{R})\rightarrow0,\,\ t\rightarrow-\infty\right\}  .
\end{array}
\]

\begin{theorem}
\label{StructureStrong} Under conditions (\ref{2}) equality (\ref{Str2}) holds.
\end{theorem}

\begin{proof}
By (\ref{PropStrong1}) we have%
\[
M_{s}^{+}(\mathfrak{R})=\left\{  z\,:\,\exists\gamma(\cdot)\in\mathbb{K}%
_{s},\,\ \gamma(0)=z,\,\,\,\ \mathrm{dist}_{H_{0}^{1}(\Omega)}(\gamma
(t),\mathfrak{R})\rightarrow0,\,\ t\rightarrow-\infty\right\}  ,
\]
and then equality (\ref{Str2}) follows from Lemmas \ref{EqAttr}, \ref{EqK} and
Theorem \ref{StructureReg}.
\end{proof}

\bigskip

\textbf{Acknowledgments.}

Partially supported by spanish Ministerio de Ciencia e Innovaci\'{o}n and
FEDER, projects MTM2011-22411 and MTM2009-11820, the Consejer\'{\i}a de
Innovaci\'{o}n, Ciencia y Empresa (Junta de Andaluc\'{\i}a) under the Proyecto
de Excelencia P07-FQM-02468 and the Consejer\'{\i}a de Cultura y Educaci\'{o}n
(Comunidad Aut\'{o}noma de Murcia), grant 08667/PI/08.

\end{document}